%% file: main.tex
\documentclass[onefignum,onetabnum]{siamart220329}

\input{shared}

\begin{document}

\maketitle

\begin{abstract} This paper studies well-posedness and parameter sensitivity of
  the \emph{Square Root LASSO (\srlasso{})}, an optimization model for
  recovering sparse solutions to linear inverse problems in finite dimension. An
  advantage of the \srlasso{} (\eg{} over the standard LASSO) is that the
  optimal tuning of the regularization parameter is robust with respect to
  measurement noise. This paper provides three point-based regularity conditions
  at a solution of the \srlasso{}: the \emph{weak}, \emph{intermediate}, and
  \emph{strong} assumptions. It is shown that the weak assumption implies
  uniqueness of the solution in question. The intermediate assumption yields a
  directionally differentiable and locally Lipschitz solution map (with explicit
  Lipschitz bounds), whereas the strong assumption gives continuous
  differentiability of said map around the point in question. Our analysis leads
  to new theoretical insights on the comparison between \srlasso{} and LASSO
  from the viewpoint of tuning parameter sensitivity: noise-robust optimal
  parameter choice for \srlasso{} comes at the ``price'' of elevated tuning
  parameter sensitivity. Numerical results support and showcase the theoretical
  findings.
\end{abstract}



\section{Introduction}
\label{sec:introduction}

In this paper we study the \emph{Square Root \lasso{}} (\emph{\srlasso{}})
\begin{equation}\label{eq:SR-LASSO} 
\min_{x\in \R^n}\|Ax-b\| +\lambda \|x\|_1,
\end{equation}
which was introduced in~\cite{belloni2011square} as an optimization model for
computing sparse solutions to the linear inverse problem $Ax\approx b$.  Here,
$A \in \mathbb{R}^{m \times n}$ is a \emph{design} or \emph{sensing matrix},
$b \in \mathbb{R}^m$ is a vector of \emph{observations} or \emph{measurements},
$\lambda >0$ is a \emph{tuning parameter}, and $\|\cdot\|$ and $\|\cdot\|_1$
denote the Euclidean and the $\ell_1$-norm, respectively. We refer to $\|Ax-b\|$
and $\lambda \|x\|_1$ as the data fidelity and the regularization term,
respectively. Seeking an optimal balance between data fidelity and
regularization, the \srlasso{} is a powerful sparse regularization technique,
widely adopted in statistics and increasingly popular in scientific computing
and machine learning (see~\cref{sec:motivation}).  The \srlasso{} is a close
relative of the well-known \emph{\lasso{}} (\emph{Least Absolute Shrinkage and
  Selection Operator})~\cite{tibshirani1996regression}, whose unconstrained
formulation is obtained from~\cref{eq:SR-LASSO} by squaring (and, optionally,
rescaling) the data fidelity term, thus also explaining the terminology.

This seemingly minor algebraic transformation corresponds to a major benefit for
the \srlasso{}: optimal tuning strategies for the parameter $\lambda$ are robust
to unknown errors (\ie{} noise) corrupting the observations
(see~\cite{adcock2019correcting, belloni2011square, van2016estimation}
and~\cref{fig:motivation-known}). This is a key practical advantage
of~\cref{eq:SR-LASSO}. For example, in the context of sparse recovery, when $A$
and $b$ arise from noisy linear measurements of a sparse or compressible vector
$x^\sharp$, \ie{} $b = Ax^\sharp + e$ ($e \in \mathbb{R}^m$), $x^\sharp$ can be
successfully recovered via the \srlasso{} for values of $\lambda$ independent of
the noise $e$ under mild conditions on $A$ (such as the robust null space
property; see \eg{}~\cite[Theorem 6.29]{adcock2022sparse} for more details). In
contrast, an order-optimal choice of tuning parameter for \lasso{} is sensitive
to the noise scale~\cite{bickel2009simultaneous,shen2015stable}. This
attractive property has made the \srlasso{} increasingly popular over the last
decade in a variety of contexts beyond statistics, such as compressed sensing,
high-dimensional function approximation, and deep learning
(see~\cref{sec:motivation}).

In this paper, building upon a line of work initiated by the
authors~\cite{berk2023lasso}, we inspect the \srlasso{} through the lens of
\emph{variational analysis}~\cite{dontchev2014implicit, Mor06, Mor18,
  rockafellar1998variational}, which leads to a full picture of
well-posedness and stability of the solution mapping of~\cref{eq:SR-LASSO}.

\subsection{Motivation}
\label{sec:motivation}

The \srlasso{} was initially proposed by Belloni
\etal{}~\cite{belloni2011square} as a sparse high-dimensional linear regression
technique. Since then, it has had a significant impact in the statistical
community, \eg{} see~\cite{belloni2014pivotal, bunea2013group, pham2015robust,
  stucky2017sharp, tian2018selective} and the book~\cite{van2016estimation}. The
\srlasso{} is also closely related to other statistical estimation techniques,
such as the \emph{scaled \lasso{}}~\cite{sun2012scaled}\cite[Chapter
3]{van2016estimation} and \emph{SPICE} (\emph{SParse Iterative Covariance-based
  Estimation})~\cite{babu2012spectral,babu2014connection,stoica2010new}.

On top of its impact in statistics, the \srlasso{} (and its \emph{weighted}
formulation, where the $\ell_1$-norm in the regularization term
of~\cref{eq:SR-LASSO} is replaced with a weighted $\ell_1$-norm) has been
gaining increasing popularity in other fields such as compressed
sensing~\cite{donoho2006compressed}, high-dimensional function approximation,
and deep learning. The (weighted) \srlasso{} was applied and studied in the
compressed sensing context by Adcock \etal{}~\cite{adcock2019correcting},
motivated by applications to high-dimensional function approximation and
parametric differential equations~\cite{adcock2022sparse}. Further studies and
applications of the \srlasso{} in compressed sensing
include~\cite{adcock2022log, adcock2021compressive, foucart2023sparsity,
  mohammad2023greedy, petersen2022robust}.  In addition, the \srlasso{} was
recently employed to analyze and develop deep learning techniques. Training
strategies based on the \srlasso{} were used to prove so-called practical
existence theorems for deep neural networks~\cite{adcock2020deep, adcock2021gap}
and to develop stable and accurate neural networks for image
reconstruction~\cite{colbrook2022difficulty}. A thorough variational analysis
has, to the best of our knowledge, been lacking thus far.

\subsection{Main contributions}
\label{sec:contributions}

Our first contribution concerns well-posedness of the \srlasso{}. Concretely,
given a solution $\bar{x}$ of~\cref{eq:SR-LASSO} with data $(A, b,\lambda)$, we
establish (in \Cref{th:sufficient-condition-for-uniqueness}) that $\bar x$ (with
support $I$) is the unique minimizer if the following holds:
\begin{assumption}[Weak]
  \label{ass:weak-1}
  We have:
  \begin{itemize}
  \item[(i)] $\ker A_I=\{0\}$ and $ b\notin \rge A_I$; 
  \item[(ii)]
    $\exists z\in \ker A_I^T\cap\left\{\frac{ b-A\bar x}{\|A\bar x- b\|}\right\}^\perp: \left\|A_{I^C}^{\top}\left(\frac{ b-A\bar x}{\|A\bar x-b\|}+z\right)\right\|_\infty<\lambda$.
  \end{itemize}
\end{assumption}
We then introduce two stronger regularity conditions for the \srlasso{}: the
first, which we call the \emph{intermediate} condition, reads as follows for
some solution $\bar x$ and
$J := \set{i \in \{1,\dots,n\}}{\left| A_i^T\frac{ b-A\bar x}{\|A\bar x- b\|}
  \right| = \lambda}$.
\begin{assumption}[Intermediate]
  \label{ass:intermediate-1}
  We have $\ker A_J=\{0\}$, $A\bar x\neq b$, and $ b\notin \rge A_J$.
\end{assumption}
We show in~\cref{prop:intermed-implies-weak} that this condition implies
\cref{ass:weak-1}. On the other hand, we find in~\cref{prop:StrongInter} that it
is implied by the \emph{strong} condition at $\bar x$.
\begin{assumption}[Strong]
  \label{ass:Strong-1}
  We have:
  \begin{itemize}
  \item[(i)] $\ker A_{I} = \{ 0\}$ and $b\notin\rge A_I$; 
  \item[(ii)] $\|A_{I^C}^{\top}(b-A\bar x)\|_\infty<\lambda\|A\bar x-b\|$.
\end{itemize}
\end{assumption}
This analysis on uniqueness and the study of the relationships between the different regularity conditions relies heavily on classical convex analysis~\cite{rockafellar1970convex}.

Our second main contribution concerns sensitivity of solutions
of~\cref{eq:SR-LASSO} to the data, \ie{} we investigate the (solution) mapping
$S : \reals^{m} \times \reals_{++} \rightrightarrows \reals^{n},$
\begin{align}
  \label{eq:solution-mapping-b-lambda}
  S : (b, \lambda) \mapsto \argmin_{x \in \reals^{n}} \left\{ \| Ax - b \| + \lambda \| x \|_{1} \right\}.
\end{align}
To this end, we bring to bear the powerful machinery of variational analysis and
the set-valued implicit function theorems built around graphical differentiation
\`a la Rockafellar and Wets~\cite{rockafellar1998variational},
Mordukhovich~\cite{Mor06, Mor18}, and Dontchev and
Rockafellar~\cite{dontchev2014implicit}.

We show in \Cref{thm:directional-differentiability} that $S$ is locally
Lipschitz (hence single-valued) and directionally differentiable at
$(\bar b,\bar \lambda)$ if \cref{ass:intermediate-1} holds at
$\bar x:=S(\bar b,\bar \lambda)$. Complementing this, in \cref{prop:Explicit} we
furnish an analytic expression for the (unique) solution under said intermediate
condition. Moreover, in \cref{thm:main-theorem}, we show that $S$ is
continuously differentiable at $(\bar b,\bar \lambda)$ if \cref{ass:Strong-1}
holds at $\bar x$. These theoretical findings are summarized in
\cref{fig:Summary-new}.

\vspace{-5mm}
\tikzcdset{every label/.append style = {font=\small,yshift=1pt}, every arrow/.append style = {shorten=2em}}
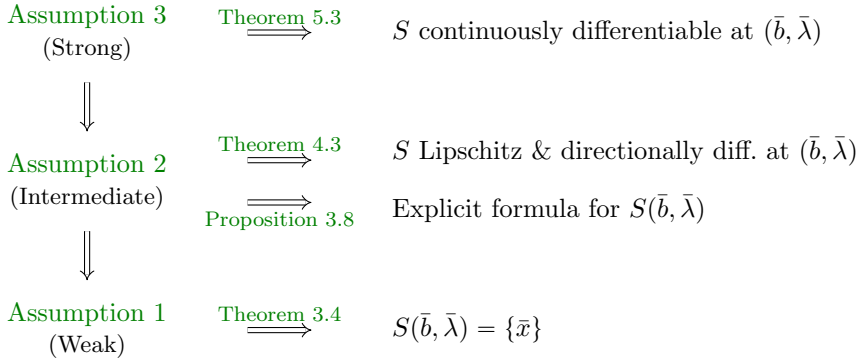
\begin{figure}[h!]
\begin{equation*}
\begin{tikzcd}[row sep=2em, column sep=15ex, /tikz/column 2/.append style={anchor=base west}] 
  \begin{tabular}[c]{c}
    \cref{ass:Strong-1} \\ {\small(Strong)}
  \end{tabular}
  \arrow[d,Rightarrow,shorten=3pt]
  \arrow[r,Rightarrow,"\text{\footnotesize\cref{thm:main-theorem}}"] &
  \begin{tabular}[c]{c}
    $S$ continuously differentiable  at $(\bar b,\bar \lambda)$
  \end{tabular}
  \\
  \begin{tabular}[c]{c}
    \cref{ass:intermediate-1}\\ {\small (Intermediate)}
  \end{tabular}
  \arrow[r, Rightarrow,shift left=.8em,"{\footnotesize\text{\cref{thm:directional-differentiability}}}"]
  \arrow[r, Rightarrow, swap, shift right=.8em, "\text{\footnotesize\cref{prop:Explicit}}"{yshift=-2pt}]
  \arrow[d, Rightarrow,shorten=3pt]
  &
  \begin{tabular}[c]{l}
    $S$ Lipschitz \& directionally
    diff.\ at $(\bar b,\bar \lambda)$\\[1em]
    Explicit 
    formula for $S(\bar{b}, \bar{\lambda})$
  \end{tabular}
  \\
  \begin{tabular}[c]{c}
    \cref{ass:weak-1}\\ {\small (Weak)}
  \end{tabular}
  \arrow[r, Rightarrow, "\text{\footnotesize\cref{th:sufficient-condition-for-uniqueness}}"] &
  \begin{tabular}[c]{c}
    $S(\bar{b},\bar{\lambda}) = \{\bar{x}\}$
  \end{tabular}
\end{tikzcd}
\end{equation*}
\caption{Ordering of regularity assumptions and their implications}
\label{fig:Summary-new}
\end{figure}

\vspace{-5mm}
\noindent
Our third main contribution is a comparison of \srlasso{} and (unconstrained)
\lasso{} (\emph{cf.}~\cref{eq:unconstrained-LASSO}). It is well known that
\srlasso{} optimal parameter choice is noise scale robust, but not so for
\lasso{} (\emph{cf.}~\cref{fig:motivation-known}). However, we elaborate
in~\cref{sec:sr-vs-uc} on the differences in sensitivity between the two
programs and suggest the ``price'' for robustness is increased parameter
sensitivity. For instance, \cref{fig:motivation-new} portrays the (local) Lipschitz
behavior of both programs, displaying elevated parameter sensitivity for
\srlasso{}. A theoretical argument supporting this behavior is given
in~\cref{sec:theoretical-comparison}
(see~\cref{eq:Lipschitz_lambda_SR-LASSO,eq:Lipschitz_lambda_LASSO}). Our insights on this robustness-sensitivity trade off for SR-LASSO's parameter tuning strategies are, to the best of our knowledge, a novel contribution.
\begin{figure}[t]
  \centering
  \begin{subfigure}[t]{.49\linewidth}
    \includegraphics[width=\linewidth]{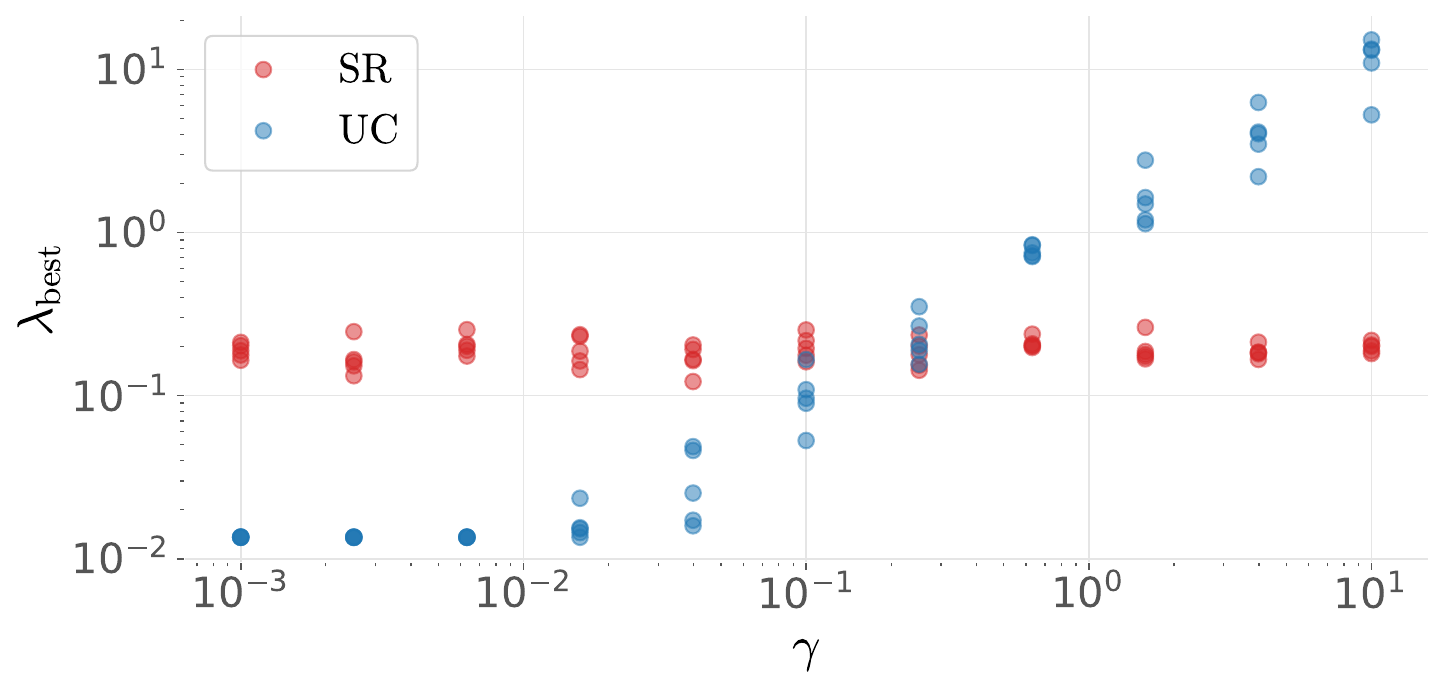}
    \caption{$\lambda_{\text{best}}^{\text{SR}}$ and
      $\lambda_{\text{best}}^{\text{UC}}$ \vs{} noise scale  $\gamma$, $5$ realizations each;
      $(m, n, s) = (50, 100, 5)$.\label{fig:motivation-known}}
  \end{subfigure}
  \hfill
  \begin{subfigure}[t]{.49\linewidth}
    \includegraphics[width=\linewidth]{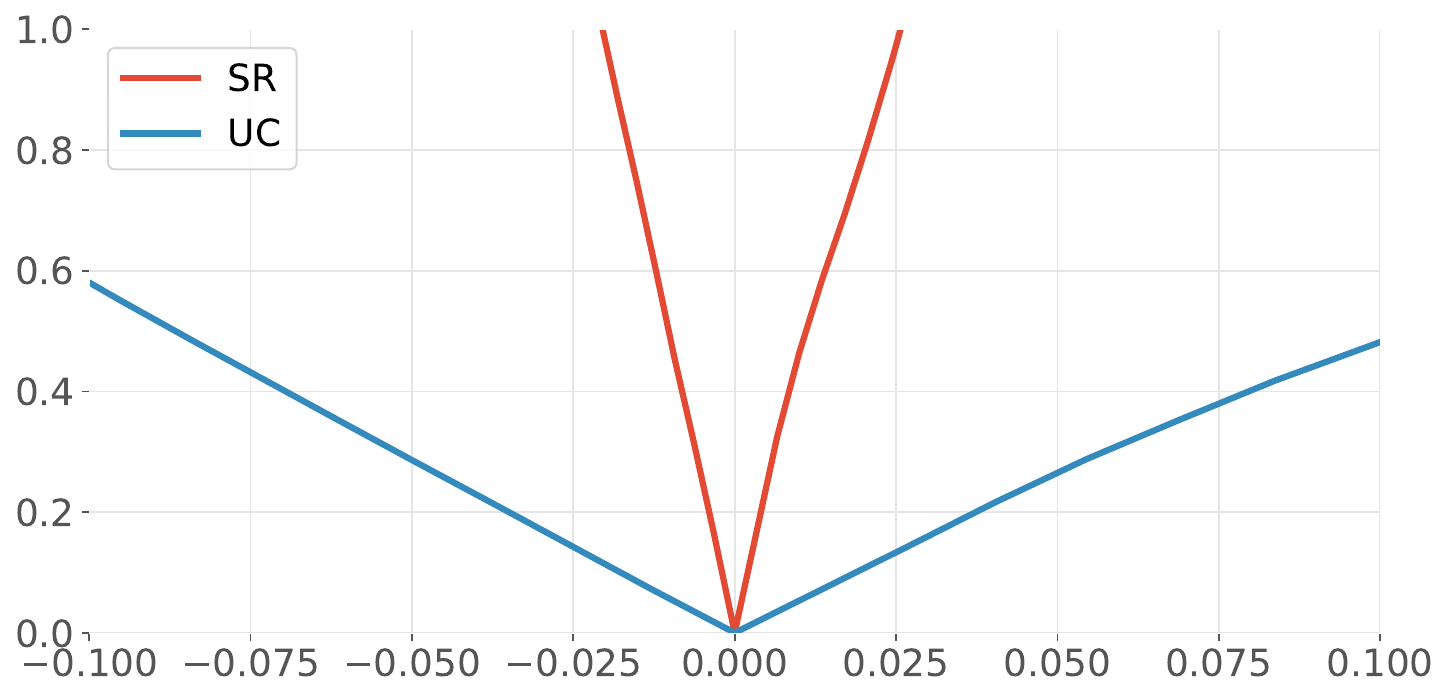}
    \caption{(Local) Lipschitz behavior for each program: $|\lambda - \bar{\lambda}_{\text{best}}|$ vs.\ $\|\bar{x}(\lambda) - \bar{x}(\bar{\lambda}_{\text{best}})\|$.\label{fig:motivation-new}}
  \end{subfigure}
  \caption{Comparison between \srlasso{} (SR) and (unconstrained) \lasso{} (UC):
    recovery of an unknown \emph{sparse} signal $x^{\sharp} \in \reals^n$ from
    noisy underdetermined linear measurements (\emph{cf.}~\cref{eq:CS_model}).
    For $\lambda > 0$, denote by $\bar{x}_{\text{SR}}(\lambda)$,
    $\bar{x}_{\text{UC}}(\lambda) \in \reals^{n}$ respective solutions to
    \srlasso{} and \lasso{}. The optimal parameter values
    $\lambda_{\text{best}}^{\text{SR}}, \lambda_{\text{best}}^{\text{UC}} > 0$
    for each program minimize the $\ell_2$ approximation error between the
    solution and $x^\sharp$. \cref{fig:motivation-new} corresponds with $(\gamma, m) = (0.5, 100)$ in \cref{fig:empirical-lipschitz-bound-mg}. See~\cref{sec:sr-vs-uc} and \cite{adcock2019correcting, belloni2011square, van2016estimation} for further details and
    discussion; 
    \cref{sec:implementation-details} for numerical implementation details.}
  \label{fig:motivation-sr-uc-comparison}
\end{figure}

Our final contribution, in~\cref{sec:numerics}, is a numerical exploration of
\srlasso{} solution uniqueness and sensitivity, as well as an empirical
verification of the tightness of our Lipschitz bounds in
\cref{sec:continuous-differentiability-solution-function} using synthetic
experiments. In particular,
\cref{fig:uniqueness-sufficiency-simple,fig:uniqueness-sufficiency-heatmap}
demonstrate a wide neighborhood in which the sufficient condition for
uniqueness is empirically satisfied. Moreover,
\cref{fig:empirical-lipschitz-bound} supports the notion that our theoretical bounds on the Lipschitz constant for \srlasso{} are relatively tight (at least in the regime considered in the experiment).

\subsection{Related work}

The well-posedness study presented in this paper is inspired by the one for the
LASSO problem in Zhang \etal{}~\cite{zhang2015necessary},
Gilbert~\cite{gilbert2017solution} and Tibshirani~\cite{tibshirani2013lasso}, as
well as the one for nuclear norm minimization by Hoheisel and
Paquette~\cite{HoP23}. The stability analysis executed here is similar to a study on the LASSO problem
carried out by the authors of this paper~\cite{berk2023lasso}. Stability
analysis for linear least-squares problems (\ie{} quadratic fidelity term) with
general, partially smooth regularizers can be found in the body of work by
Vaiter \etal{}, \eg{}~\cite{VDF15, VDF17}. This embeds in the more
general and more recent studies (which do \emph{not} cover the \srlasso{}) by
Bolte \etal{}~\cite{bolte2024differentiating, BPS21}. A sensitivity analysis of the proximal
operator using tools similar to our study can be found in Friedlander \etal{}~\cite{friedlander2022perspective}. Tuning parameter sensitivity has
previously been examined for other LASSO formulations~\cite{berk2021best} and
for proximal denoising~\cite{berk2021sensitivity}.

Further studies that tackle the \srlasso{} explicitly, albeit not from a
variational-analytic perspective, were discussed in \cref{sec:motivation}. They
include contributions in the fields of statistics~\cite{belloni2011square,
  belloni2014pivotal, bunea2013group, pham2015robust, stucky2017sharp,
  tian2018selective, van2016estimation}, sparse recovery, compressed
sensing~\cite{adcock2019correcting, adcock2022sparse, adcock2022log,
  foucart2023sparsity, mohammad2023greedy, petersen2022robust}, and deep
learning~\cite{adcock2020deep, adcock2021gap, colbrook2022difficulty}.

\subsection{Notation}
\label{sec:notation}

In what follows, the Euclidean norm is denoted by $\|\cdot\|$, the
$\ell_{1}$-norm (on $\R^n$) is given by $\|x\|_1 := \sum_{i=1}^n|x_i|$, while
the $\ell_\infty$-norm (or \emph{maximum norm}) is given by
$\|x\|_\infty:=\max_{i=1,\dots,n}|x_i|$. The corresponding unit balls have the
respective subscripts, \ie{} $\bB$, $\bB_1$ and $\bB_\infty$.  The support of a
vector $x\in\R^n$ is given by $\supp(x):=\set{i\in \{1,\dots,n\}}{x_i\neq
  0}$. The first $n$ positive integers are denoted $[n] := \{1, 2, \ldots,
n\}$. Define the projection operator onto a closed, convex set
$C \subseteq \reals^{n}$ by
$\proj_{C}(x) := \argmin_{z \in C} \| z - x \|$. Write the orthogonal
complement of a subspace $V \subseteq \reals^{m}$ as $V^{\perp}$. The identity
matrix is denoted by $\identity$. The spectral norm of a matrix
$X\in \R^{m\times n}$ is denoted by $\|X\|$.  For a set $K \subseteq [n]$,
$X_M\in \R^{m\times |K|}$ is the matrix whose columns are the columns $X_i$ of
$X$ for $i\in K$. For $x \in \reals^{n}$, we let $L_{K}(x) \in \reals^{n}$ be
the vector whose elements are $x_{i}$ if $i \in K$ and $0$ otherwise.

\section{Preliminaries}
\label{sec:preliminaries}

We start with some basic results from matrix analysis. Denote the (Moore-Penrose) pseudoinverse of $X$ by $X^{\dagger}$. For a symmetric matrix
$S\in \R^{d\times d}$, let $\lambda_{\max}(S)$ be its maximum eigenvalue. For a
matrix $X\in \R^{m\times s}$, let $\sigma_{\min}(X)$ be its smallest nonzero
singular value, and let $\sig_{\max}(X)$ be its maximum singular value. Recall
that
\[
\sigma_{\max}(X)=\|X\|=\sqrt{\lambda_{\max}(X^{T}X)}=\max_{v\in \bB}v^{T}Xv.
\]

We commence with a result that is ubiquitous in our study and is, in essence, the famous Sherman-Morrison-Woodbury formula~\cite[(0.7.4.1)]{horn2012matrix}.
\begin{lemma}[Sherman-Morrison-Woodbury]
  \label{lem:SM}
  Let $M\in \R^{m\times s}$ such that $\rank M=s$, and let $v\in \R^m$ with
  $\|v\|=1$. For the matrix $W:=M^{\top}(\identity-vv^{\top})M$ the following
  hold:
  \begin{itemize}
  \item[(a)] $W$ is invertible (in fact, symmetric positive definite) if (and
    only if) $v\notin \rge M$ with
    \[
      W^{-1}=(M^{\top}M)^{-1}+\frac{M^\dagger v(M^\dagger v)^{\top}}{1-v^{\top}MM^\dagger v}.
    \]
  \item[(b)] In the invertible case, we have
    \[
      \lambda_{\max}(W^{-1})%
      \leq \frac{1}{\sigma_{\min}(M)^2}+ \frac{\|M^\dagger v\|^2}{1-v^{\top}MM^\dagger v} %
      < \frac{1}{\sigma_{\min}(M)^2}+ \frac{1}{1-v^{\top}MM^\dagger v}.
    \]
  \end{itemize}
\end{lemma}
\begin{proof}
  (a) First, observe that with the invertible matrix $A := M^{\top}M$,
  $x:=-M^{\top}v$, and $y := -x$, we have
  \begin{eqnarray*}
    1+x^{\top}A^{-1}y=0 & \iff & 1= v^{\top}M(M^{\top}M)^{-1}M^{\top}v=v^{\top}MM^\dagger v=\|MM^\dagger v\|^2\\
                        & \iff & \|MM^\dagger v\|=\|v\|\\
                        & \iff & v\in \rge M.
  \end{eqnarray*}
  The last equivalence follows from the fact that $MM^\dagger$ is the projection
  onto $\rge M$ and the Cauchy-Schwarz inequality.  Hence, by the
  Sherman-Morrison-Woodbury formula~\cite[(0.7.4.1)]{horn2012matrix} with
  $A, x, y$ as above, and assuming that $v\notin \rge M$, we have
  \begin{eqnarray*}
    W^{-1}  %
    = (M^{\top}M)^{-1} + \frac{(M^{\top}M)^{-1}M^{\top}vv^{\top}M(M^{\top}M)^{-1}}{1-v^{\top}M(M^{\top}M)^{-1}M^{\top}v}
    =   (M^{\top}M)^{-1}+\frac{M^\dagger v(M^\dagger v)^{\top}}{1-v^{\top}MM^\dagger v}.
  \end{eqnarray*}
  \smallskip
  
  \noindent
  (b) By \emph{Weyl's theorem}~\cite[Theorem 4.3.1]{horn2012matrix}, it follows from (a) that
  \begin{eqnarray*}
    \lambda_{\max}(W^{-1}) %
    & \leq %
    &  \lambda_{\max}((M^{\top}M)^{-1})+\lambda_{\max}\left( \frac{M^\dagger v(M^\dagger v)^{\top}}{1-v^{\top}MM^\dagger v}\right)
    \\
    & = & \frac{1}{\sigma_{\min}(M)^2}+\frac{\|M^\dagger v\|^2}{1-\|MM^\dagger v\|^2}.
  \end{eqnarray*}
  \hphantom{a}
\end{proof}

\noindent
We record an immediate consequence. 

\begin{corollary}
  \label{cor:Equiv}
  Let $M\in \R^{m\times s}$ and let $v\in \R^s$ with $\|v\|=1$. Then the
  following are equivalent:
  \begin{itemize}
  \item[(i)] $\ker M=\{0\}$ and $v\notin \rge M$;
  \item[(ii)] $\ker M^{\top}(\identity - vv^{\top})M = \{0\}$;
  \item[(iii)] $\ker (I-vv^{\top})M=\{0\}$.
  \end{itemize}
\end{corollary}
\begin{proof}
  \vphantom{a}

  \noindent{}`$(i)\Rightarrow (ii)$': This follows immediately from \Cref{lem:SM}(a).

  \noindent{}`$(ii) \Rightarrow (iii)$': This is obvious.

  \noindent{}`$(iii) \Rightarrow (i)$': Assume (i) were false. Then $\ker M$ is
    nontrivial, in which case so is $\ker (\identity - vv^{\top})M$, or
    $v\in \rge M$. The latter, however, implies that there exists
    $y\in \R^s\setminus\{0\}$ such that
    $ (\identity - vv^{\top})My = (\identity - vv^{\top})v = 0, $ hence, also in
    this case, $\ker (\identity-vv^{\top})M$ is nontrivial.
\end{proof}

\subsection{Tools from  variational analysis}
\label{sec:tools-variational-analysis}
We provide in this section the necessary tools from variational analysis, and we
follow here the notational conventions of Rockafellar and
Wets~\cite{rockafellar1998variational}, but the reader can find the objects
defined here also in the books by Mordukhovich~\cite{Mor06, Mor18} or
Dontchev and Rockafellar~\cite{dontchev2014implicit}.

Let $S:\bE_1\rightrightarrows\bE_2$ be a set-valued map. The domain and graph of $S$, respectively, are $\dom S:=\set{x \in \bE_1}{S(x)\neq \emptyset}$ and $\gph S:=\set{(x,y)\in \bE_1\times \bE_2}{y\in S(x)}$. The \emph{outer limit} of $S$ at $\bar x \in \bE_1$  is
\[
\Limsup_{x\to \bar x} S(x):=\set{y\in\bE_2}{\exists \{x_k\}\to \bar x,  \{y_k\in S(x_k)\}\to y}.
\]
Now let $A\subseteq  \bE$.  The \emph{tangent  cone} of $A$ at $\bar x\in A$ is 
$
T_A(\bar x):=\Limsup_{t\downarrow 0} \frac{A-\bar x}{t}.
$
The \emph{regular normal cone} of $A$ at $\bar x\in A$ is the polar of the tangent cone, \ie{}
\[
\hat N_A(\bar x):=T_A(\bar x)^\circ=\set{v \in \bE}{\ip{v}{y}\leq 0\;\;\forall y\in T_{A}(\bar x)}.
\]
The \emph{limiting normal cone}   of $A$ at $\bar x\in A$ is 
$
N_A(\bar x):=\Limsup_{x\to \bar x} \hat N_A(x).
$
The \emph{coderivative} of $S$ at $(\bar x,\bar y)\in \gph S$ is the map $D^*S(\bar x\mid\bar y):\bE_2\rightrightarrows\bE_1$ defined via
\begin{equation}
\label{eq:def_coderivative}
v\in D^*S(\bar x\mid\bar y)(y) \IFF (v,-y)\in N_{\gph S}(\bar x,\bar y).
\end{equation}
The    \emph{graphical derivative} of $S$ at $(\bar x,\bar y)$ is the map $DS(\bar x\mid\bar y):\bE_1\rightrightarrows \bE_2$ given by
\begin{equation}
\label{eq:def_graphical_derivative}
v\in DS(\bar x \mid \bar y)( u) \IFF (u,v) \in T_{\gph S}(\bar x,\bar y).
\end{equation} 
The \emph{strict graphical derivative} of $S$ at $(\bar x,\bar y)$ is  $D_*S(\bar x\mid\bar y):\bE_1\rightrightarrows \bE_2$ given by 
\begin{eqnarray*}
D_*S(\bar x\mid\bar y)(w)=
 \set{z \in \bE_1}{\exists \left\{\begin{array}{l}
 \vspace{0.1cm}\{t_k\}\downarrow 0,\, \{w_k\}\to w,\\ 
 \{z_k\}\to z,\\
  \{(x_k,y_k)\in\gph S\}\to (\bar x,\bar y) \end{array}\right\}: z_k\in\frac{S(x_k+t_kw_k)-y_k}{t_k}}.
\end{eqnarray*}
We adopt the convention to set $D^*S(\bar x):=D^*S(\bar x\mid\bar y)$ if
$S(\bar x)$ is a singleton, and proceed analogously for the graphical
derivatives. We point out that if $S$ is single-valued and continuously
differentiable at $\bar x$, then $DS(\bar x)=D_*S(\bar x)$ coincides with its
derivative at $\bar x$. Moreover, in this case $D^*S(\bar x)=DS(\bar{x})^*$. Therefore, there is, in this case, no ambiguity in notation. More
generally, we will employ the following sum rule for the derivatives introduced
above frequently in our study.

\begin{lemma}[{\cite[Exercise 10.43 (b)]{rockafellar1998variational}}]\label{lem:Sum} Let $S=f+F$ for $f:\bE_1\to\bE_2$, $F:\bE_1\rightrightarrows \bE_2$, and $(\bar x,\bar u)\in \gph S$. Assume that $f$ is continuously differentiable at $\bar x$. Then:
\begin{itemize}
    \item [(a)] $DS(\bar x|\bar u)(w)=Df(\bar x)w+DF(\bar x|\bar u-f(\bar x))(w),\quad \forall w\in \bE_1$;
    \item[(b)]  $D_*S(\bar x|\bar u)(w)=Df(\bar x)w+D_*F(\bar x|\bar u-f(\bar x))(w),\quad  \forall w\in \bE_1$;
    \item[(c)]  $D^*S(\bar x|\bar u)(y)=Df(\bar x)^*y+D^*F(\bar x|\bar u-f(\bar x))(y),\quad \forall y\in \bE_2$.
\end{itemize}
\end{lemma}

\subsection{Convex analysis tools} 
\noindent
We first present some fundamental concepts associated with convex sets \cite{rockafellar1970convex}.
For a convex  set $C\subseteq \R^n$  and $\bar x\in C$ we define:
\begin{itemize}
\item[(a)] the \emph{affine hull} of $C$, \ie{} is the smallest affine set that contains $C$, by $\aff C$;

\item[(b)] the \emph{subspace parallel to $C$} by 
$
\para C:=\lin(C-\bar x)=\aff C -\bar x;
$
\item[(c)] its \emph{relative interior} by  
$
\ri C:=\set{x\in C}{\exists \varepsilon>0: \; B_{\varepsilon}(x)\cap \aff C\subseteq C}.
$ 
\end{itemize}
We note that $\inter C\neq \emptyset$ if and only if $\para C=\R^n$, in which case $\ri C=\inter C$. We call a function $f:\R^n\to \rp$ \emph{proper} if  $\dom f:=\set{x\in \R^n}{f(x)<+\infty}$ is nonempty. It is called  \emph{convex} if its \emph{epigraph} $\epi f:=\set{(x,\alpha)\in \R^n\times \R} {f(x)\leq \alpha}$ is a convex set. As a first, yet central, example of an  extended real-valued function, we consider the \emph{indicator (function)} $\delta_C:\R^n\to \rp$ of  $C\subseteq \R$ given by
\[
\delta_C(x)=
\begin{cases}0& x\in C\\
+\infty& \text{else},
\end{cases}
\]
which is proper and convex if and only if $C$ is nonempty and convex. The \emph{(convex) subdifferential} of $f:\R^n\to \rp$ at $\bar x\in \dom f$ is given by
\[
\p f(\bar x)=\set{v\in \R^n}{f(\bar x)+\ip{v}{x-\bar x}\leq f(x)\;\forall x\in \dom f}.
\]
For instance, the subdifferential of the indicator function of a convex set $C\subseteq \R^n$  at $\bar x\in C$  is the normal cone of $C$ at $\bar x$, \ie{}
\[
\p\delta_C(\bar x)=\set{v\in \R^n}{\ip{v}{x-\bar x}\leq 0\quad \forall x\in C}=N_C(\bar x).
\]
The two most important examples to our study  are the Euclidean norm $\|\cdot\|$ whose subdifferential is given by
 \begin{equation}\label{eq:SD2Norm} 
 \p \|\cdot\|(z)=\begin{cases}\left\{\frac{z}{\|z\|}\right\} & z\neq 0\\ \bB & z=0,
 \end{cases}
\end{equation}
and the $\ell_1$-norm whose subdifferential is presented in the next result.

\begin{lemma}[Subdifferential of $\ell_1$-norm]
  \label{lem:SD1Norm}
  Let $z\in \R^n$ and set $I:=\supp(z):$
  \begin{itemize}
  \item[(a)] $\p\|\cdot\|_1(z)=\bigtimes_{i=1}^n \left.\begin{cases} \sgn(z_i),& z_i\neq 0,\\
                                                         [-1,1],& z_i=0
                                                       \end{cases}\right\}$;

  \item[(b)] $\ri  \p\|\cdot\|_1(z)=\bigtimes_{i=1}^n\left.\begin{cases} \sgn(z_i),& z_i\neq 0,\\
(-1,1),& z_i=0
\end{cases}\right\}$;

  \item[(c)] $\para \p\|\cdot\|_1(z)=\set{v\in \R^n}{v_{I}=0}$.
\end{itemize}

\end{lemma}

\noindent
The \emph{(Fenchel) conjugate} of $f:\R^n\to \rp$ is the function $f^*:\R^n\to \rp$,
\[
f^*(y):=\sup_{x\in \dom f}\{\ip{y}{x}-f(x)\}.
\]
If $f$ is proper and convex, then so is $f^*$ (which is always lower
semicontinuous).  Of special interest is the conjugacy relation between
indicator and support functions. For any set $C\subseteq \R^n$, its support
function is $\sig_C(y):= \sup_{x\in C}\ip{x}{y}=\delta^*_C(y)$.  We have
$\sigma_C^*(y)=\delta_{C}(x)$ if (and only if) $C$ is nonempty, closed and
convex.  This is relevant to our study as every norm is the support function of
a symmetric, convex, compact set $C$ with $0\in \inter C$. In particular,
$\|\cdot\|_p=\sig_{\bB_q}$ for all $1\leq p,q\leq\infty$ and
$\frac{1}{p}+\frac{1}{q}=1$\footnote{Formally setting $1/\infty:=0$.}, eg{}
see~\cite[11(12)]{rockafellar1998variational}. Consequently, we have
\begin{equation}\label{eq:NormConj}
\|\cdot\|^*=\delta_\bB \AND \|\cdot\|_1^*=\delta_{\bB_\infty}.
\end{equation}

\section{Uniqueness of solutions  and regularity conditions for \srlasso{}}
\label{sec:uniqueness-solutions}

This section establishes a sufficient condition for \srlasso{} solution uniqueness, then introduces two stronger point-based regularity conditions and their relationship.  

\subsection{Uniqueness of solutions}

In this section, we provide conditions that guarantee a given solution of
\srlasso{} is unique. The key observation is that the normalized residual is, in
essence, an invariant for a given problem instance. This fact can be seen
through (Fenchel-Rockafellar) duality (\eg{}
see~\cite[Chapter~11]{rockafellar1998variational}), which we establish for the
\srlasso{}~\cref{eq:SR-LASSO} here.

\begin{proposition}[Fenchel-Rockafellar duality scheme for~\cref{eq:SR-LASSO}]
  \label{prop:FR}
  We have:
  \begin{itemize}
  \item[(a)] The (Fenchel-Rockafellar) dual problem to~\cref{eq:SR-LASSO} is
    \begin{equation}
      \label{eq:srlasso-dual}
      \max_{y\in \R^m}\; \ip{b}{y}\st A^{\top}y\in \lambda \bB_\infty, \quad y\in \bB.
    \end{equation}
  \item[(b)] For $(\bar x,\bar y)$ the following are equivalent:
    \begin{itemize}
    \item[(i)] $\bar x$ solves~\cref{eq:SR-LASSO}, $\bar y$ solves~\cref{eq:srlasso-dual}.
    \item[(ii)] $\|A\bar x-b\|+\lambda \|\bar x\|_1=\ip{b}{\bar y}$, and
      $\bar y$ is feasible for~\cref{eq:srlasso-dual}.
    \item[(iii)] $A^{\top}\bar y \in \lambda \p\|\cdot\|_1(\bar x)$,\;
      $-\bar y\in \p \|(\cdot)-b\|(A\bar x)$.
    \item[(iv)] $\bar x\in N_{\lambda \bB_\infty}(A^{\top}\bar y),$\;
      $b-A\bar x\in N_{\bB}(\bar y)$.
    \end{itemize}

  \end{itemize}
\end{proposition}
\begin{proof}
  Apply~\cite[Example 11.41]{rockafellar1998variational} in combination
  with~\Cref{lem:SD1Norm}, \cref{eq:SD2Norm} and~\cref{eq:NormConj}, and realize
  that strong duality holds (as the the primal functions are finite-valued),
  \ie{} primal and dual optimal value are identical, which explains the
  equivalence of (i) and (ii).
\end{proof}

\noindent
We now present the advertised invariance of the normalized residual.

\begin{corollary}
  \label{cor:Invar}
  If there exists a solution $\hat x$ of~\cref{eq:SR-LASSO} with
  $A\hat x-b\neq 0$, then:
  \begin{itemize}
  \item[(a)] The dual problem~\cref{eq:srlasso-dual} has a unique solution
    $\bar y$ with $\|\bar y\|=1$.
  \item[(b)] Every solution $\bar x$ of~\cref{eq:SR-LASSO} satisfies
    $b-A\bar x=\|A\bar x-b\|\bar y$.
  \end{itemize}
\end{corollary}
\begin{proof}
  (a) Observe that \emph{every} dual solution $\bar y$, by \Cref{prop:FR}(b),
  must satisfy
  \[
    -\bar y\in \p \|\cdot\|(A\hat x-b)=\left\{\frac{A\hat x-b}{\|A\hat x-b\|}\right\}.
  \]
  (b) Let $\bar y$ be the unique dual solution. Pick any $\bar x$ that
  solves~\cref{eq:SR-LASSO} such that $A\bar x-b\neq 0$. Then, by
  \Cref{prop:FR} (b), it follows that $\bar y=\frac{b - A\bar x}{\|A\bar x-b\|}$.
\end{proof}

\noindent
We  record a simple linear-algebraic fact.

\begin{lemma}
  \label{lem:Aux}
  Let $0\in A\subseteq\R^n$ and let $U\subseteq \R^n$ be a subspace. Then
  $\lin(U+A)=U+\lin A$.
\end{lemma}
\begin{proof}
  First, observe that $U+A\subseteq U+\lin A$, hence
  $\lin(U+A)\subseteq U+\lin A$.

  In turn, observe that $U\subseteq U+A\subseteq \lin(U+A)$ and
  $\lin A\subseteq \lin(U+A)$ (since both $U$ and $A$ contain
  $0$). Consequently, we have
  \[
    U+\lin A=\lin (U\cup \lin A)\subseteq \lin (U+A).
  \]
\end{proof}

\noindent
The main result on uniqueness relies on the following regularity condition that we impose at a solution $\bar x$ of~\cref{eq:SR-LASSO}.

\setcounter{assumption}{0}

\begin{assumption}[Weak]
  \label{ass:weak}
  For  a solution $\bar x$ of~\cref{eq:SR-LASSO} with $I := \supp(\bar x)$ we have: 
  \begin{itemize}
  \item[(i)] $\ker A_I = \{0\} \AND b \notin \rge A_I$; 
  \item[(ii)]
    $\exists z\in \left\{\frac{b-A\bar x}{\|A\bar x-b\|}\right\}^\perp \cap \ker A_{I}^{\top}: \left\|A_{I^C}^{\top} \left( \frac{b - A\bar{x}}{\|A\bar{x} - b\|} + z\right) \right\|_{\infty} < \lambda$.
  \end{itemize}
\end{assumption}

\noindent
We are now in a position to present the advertised uniqueness result. 

\begin{theorem}[Uniqueness of solutions]
  \label{th:sufficient-condition-for-uniqueness}
  Under \cref{ass:weak}, $\bar x$ is the unique solution of~\cref{eq:SR-LASSO}.
\end{theorem}
\begin{proof}
  Let $\bar y:=\frac{b-A\bar x}{\|A\bar x-b\|}$ denote the unique dual solution
  and consider the auxiliary problem with
  $\cS:=N_{\bB}(\bar y)=\R_{+}\{\bar y\}$:
  \begin{equation}
    \label{eq:R}
    \min_{x\in \R^n}\psi(x):=\lambda \|x\|_1-\ip{A^{\top}\bar y}{x}+\delta_{\cS}(b-Ax).
  \end{equation}
  The optimality conditions for~\cref{eq:R} read
 \[
   0%
   \in \p\psi(x) %
   = \lambda \p\|\cdot\|_1(x)-A^{\top}\bar y-A^{\top}N_{\cS}(b-Ax),%
   \quad b-Ax\in\cS.
 \]
 Using \cref{prop:FR}(b) and the fact that $0\in N_{\cS}(b-Ax)$, we see that
 every solution of~\cref{eq:SR-LASSO} solves~\cref{eq:R}. Now, $\bar x$ is the
 unique solution of~\cref{eq:R} if (and only if)
 $0\in \inter \p \psi(\bar x)$~\cite[Lemma 3.2]{gilbert2017solution}.  Since
 $A\bar x-b\neq 0$, we find that $ N_\cS(b-A\bar x)=\{\bar y\}^\perp=\rge P$,
 where $P:=\identity - \bar{y}\bar{y}^{\top}$ is the orthogonal projection onto
 $\{\bar y\}^\perp$. We now observe that $0\in \inter \p \psi(\bar x)$ if and
 only if the following two conditions hold:
 \[ \lin \p \psi(\bar x)=\R^n\AND 0\in \ri \p \psi(\bar x). \] Since our
 assumptions imply that $N_{\cS}(b-A\bar x)=\rge P$, we find that
 \begin{eqnarray*}
   \lin \p \psi(\bar x)=\R^n &\iff &\para \p\|\cdot\|_1(\bar x)+\rge (A^{\top}P)=\R^n\\
    &\iff &\left( \para\p\|\cdot\|_1(\bar x)\right)^\perp\cap \ker (PA)=\{0\}\\
    &\iff & \ker (PA_I)=\{0\}\\
    & \iff & \ker A_I=\{0\}\AND b\notin \rge A_I.
 \end{eqnarray*}
 Here the first identity uses \Cref{lem:Aux} and the fact that
 $A^{\top}\bar y\in \lambda\p\|\cdot\|_1(\bar x)$. The second one follows by
 taking orthogonal complements on both sides. The penultimate equivalence uses
 \Cref{lem:SD1Norm}(c), and the last equivalence uses \Cref{cor:Equiv}.

 In addition, we observe that
 \begin{eqnarray*}
   0\in \ri \p \psi(\bar x) &\iff &\exists z\in \{\bar y\}^\perp:  A_i^{\top}\bar y\in \lambda\ri\p|\cdot|(\bar x_i)- A_i^{\top}z,
 \end{eqnarray*}
 which can be seen to be equivalent to (ii) by \Cref{lem:SD1Norm} (b).
\end{proof}

\begin{remark}
  Note that for $I=\emptyset$ (\ie{} $\bar x=0$), condition (i) is vacuously
  satisfied, so the sufficient conditions collapse to
  \begin{equation}
    \label{eq:weak0}
    \exists z\in \{b\}^\perp: \|A_{I^C}^{\top}(b/\|b\|+z)\|_\infty<\lambda.
  \end{equation}
  The former corresponds to the fact that for $\bar x=0$, we have
  $\para \p\psi (\bar x)=\R^n$.\hfill$\diamond$
\end{remark}

\subsection{Stronger regularity conditions} 

\noindent
We now introduce two additional regularity conditions, both of
which (will be seen to) imply \cref{ass:weak}, and hence also guarantee
well-posedness of the \srlasso{}. As we will see in
\cref{sec:lipschitz-stability-solution-function} and
\cref{sec:continuous-differentiability-solution-function}, respectively, these
conditions in fact yield stability of the solution function.

\subsubsection{Intermediate condition}

We start with what we call the \emph{intermediate} condition. This condition is
based on the notion of the \emph{(\srlasso{}) equicorrelation set}, which is an
analog to the one in the \lasso{} setting~\cite{tibshirani2013lasso}.
\begin{assumption}[Intermediate]
  \label{ass:intermediate}
  For a minimizer $\bar x$ of~\cref{eq:SR-LASSO}, we have $A\bar x\neq b$ and
  for
  \[
    J := \set{i \in [n]}{\left| A_{i}^T\frac{b-A\bar x}{\|A\bar x-b\|}\right| = \lambda },
  \] 
  we have $\ker A_J=\{0\}$ and $b\notin \rge A_J$.
\end{assumption}
Observe that if $\bar x$ is minimizer of~\cref{eq:SR-LASSO} with
$A\bar x\neq b$, then we have $I:=\supp(\bar x)\subseteq J$ by first-order
optimality conditions; a fact that is frequently used from here on.

We will now show that \cref{ass:intermediate} implies \cref{ass:weak} as
advertised. To this end, we employ the following lemma, whose proof is deferred
to~\Cref{app:Shrinking}.

\begin{lemma}[Shrinking property]
  \label{lem:shrinking-property}
  Let
  $B \in \reals^{m \times \ell}, C \in \reals^{m \times s}, \bar{y} \in
  \reals^{m}$ and $\varepsilon > 0$ such that
  $\rank[B\;  C] = \ell + s$ and
  $\bar{y} \notin \rge[B\;  C]$. Set $\mathcal{T} := \{\bar{y}\}^{\perp} \cap \ker C^{\top}$ and 
  \begin{align*}
    p^{*} %
    := \inf_{z \in \varepsilon \mathbb{B} \cap \mathcal{T}}\|B^{\top} (\bar{y} + z) \|_{\infty}. %
  \end{align*}
  Then, $p^{*} < \| B^{\top}\bar{y} \|_{\infty}$, and the infimum is attained.
\end{lemma}

\noindent
We are now in a position to present the advertised implication.

\begin{proposition}
  \label{prop:intermed-implies-weak}
  \cref{ass:intermediate} implies \cref{ass:weak}.
\end{proposition}

\begin{proof}
  Let $\bar{x}$ solve~\cref{eq:SR-LASSO}, set $I := \supp(\bar{x})$ and
  $s := |I|$. Note that $b \notin \rge A_{J}$, thus, in particular,
  $b \notin \rge A_{I}$ and $A\bar x\neq b$, so
  $\bar y:=\frac{b-A\bar x}{\|A\bar x-b\|}$ is the unique dual solution
  (\emph{cf.}~\cref{cor:Invar}).  Thus, to establish the result, we show that
  \begin{align*}
    \exists \;   z \in \{\bar{y}\}^{\perp}\cap \ker A_I^T:\;\| A_{I^{C}}^{\top} (\bar{y} + z) \|_{\infty} < \lambda.  
  \end{align*}
  Now, set $\mathcal{T} := \{\bar{y}\}^{\perp} \cap \ker A_{I}^{\top}$. Since
  one already has $\| A_{J^{C}}^{\top}\bar{y} \|_{\infty} < \lambda$ by
  definition of $J$, choose any $\varepsilon > 0$ satisfying
  \begin{align*}
    \sup_{z \in \varepsilon \mathbb{B}\cap \mathcal{T} }\| A_{J^{C}}^{\top} (\bar{y} + z) \|_{\infty} %
    < \lambda,
  \end{align*}
  and seek $z \in \varepsilon \mathbb{B} \cap \mathcal{T}$ satisfying
  $    \| A_{J\setminus I}^{\top} (\bar{y} + z) \|_{\infty} < \lambda$.
  We select such a $z$ via~\cref{lem:shrinking-property} with
  $B := A_{J\setminus I}, C := A_{I}$ and $\ell := |J\setminus I| = |J| -
  s$. Thus, one has
  $\rank [A_{J\setminus I}\; A_{I}] = \rank A_{J} = |J|$, and
  $\bar{y} \notin \rge A_{J} = \rge [A_{J\setminus I}\; A_{I}]$. Therefore, the lemma yields
  $\bar{z} \in \varepsilon \mathbb{B} \cap \mathcal{T}$ satisfying
  $\| A_{J\setminus I}^{\top} (\bar{y} + \bar{z}) \|_{\infty} < \| A_{J\setminus I}^{\top} \bar{y} \|_{\infty} = \lambda$.
  Consequently, there exists $\bar{z} \in\cT$ satisfying
  $\| A_{I^{C}}^{\top}(\bar{y} + \bar{z}) \|_{\infty}<\lambda$, completing the proof.
\end{proof}

\noindent
An immediate consequence of the above result
and~\Cref{th:sufficient-condition-for-uniqueness} is that the intermediate
condition from~\cref{ass:intermediate} yields uniqueness of the solution in
question. This is complemented by the following result, the proof of which is
postponed to~\cref{app:Explicit}. Under~\cref{ass:intermediate}, it establishes
uniqueness and gives an analytic expression for the unique solution, analogous
to a result for unconstrained LASSO~\cite[Lemma~2]{tibshirani2013lasso}.

\begin{proposition}[Analytic solution formula]
  \label{prop:Explicit}
  Let $\bar x$ be a solution of~\cref{eq:SR-LASSO} such
  that~\cref{ass:intermediate} holds at $\bar x$. Then $\bar x$ is the unique
  solution and
  \begin{align}
    \label{eq:explicit-solution--intermediate}
    \bar{x} = L_{J} \left( B A_{J}^{\top} (\identity - \bar{y} \bar{y}^{\top}) b \right), %
    \qquad%
    B := \left[ A_{J}^{\top}(\identity - \bar{y}\bar{y}^{\top}) A_{J} \right]^{-1}.
  \end{align}
\end{proposition}

\subsubsection{The strong condition}

We present a third regularity condition to which we will refer as the `strong'
condition as it implies the intermediate condition from \cref{ass:intermediate}
(and thus the weak one) as we will see shortly.

\begin{assumption}[Strong]
  \label{ass:Strong}
  For a minimizer $\bar x$ of~\cref{eq:SR-LASSO} with
  $I := \supp(\bar x)$ we have:
  \begin{itemize}
\item[(i)] $\ker A_{I} = \{ 0\}$ and $b\notin \rge A_I$; 
\item[(ii)] $\|A_{I^C}^{\top} (b - A\bar{x})\|_{\infty} < \lambda\|A\bar{x} - b\|$.
\end{itemize}

\end{assumption}

\begin{remark}[On \cref{ass:Strong}]  Note that part (ii) is automatically satisfied if $\|A_{I^C}^\top\|_{2\to \infty}< \lambda$, where $\|\cdot\|_{2\to\infty}$ denotes the induced matrix norm defined by $\|M\|_{2\to \infty} := \sup_{\|z\| =1} \|Mz\|_\infty$. In particular, (ii) is implied by $\|A_{I^C}^\top\| < \lambda$ since $\|M\|_{2\to\infty} \leq \|M\|$ for any matrix $M$. \hfill$\diamond$
\end{remark}

\noindent
We now address the advertised (and trivial) implication.

\begin{proposition}\label{prop:StrongInter} \cref{ass:Strong} implies \cref{ass:intermediate}.
\end{proposition}
\begin{proof} If \cref{ass:Strong} holds at $\bar x$, then part (ii) yields that $I=J$, hence part (ii) implies that $A_J=A_I$ has full rank. \end{proof}

\subsubsection{Overview of regularity conditions}

From \cref{prop:intermed-implies-weak} and \cref{prop:StrongInter} it follows that: 
\vspace{0.3cm}
\begin{center}
\cref{ass:Strong} \quad $\Longrightarrow$ \quad \cref{ass:intermediate}\quad $\Longrightarrow$ \quad \cref{ass:weak}.
\end{center}
\vspace{0.3cm}

\noindent
The reverse implications do not generally hold as the following examples show.

\begin{example}
  Consider the \srlasso{} \cref{eq:SR-LASSO} with
  \begin{align*}
    A=\begin{pmatrix}
        1 & 0 &0\\0& 1 &1
      \end{pmatrix},\AND b=\binom{1}{2}.
  \end{align*}
  For $\lambda=\frac{2}{\sqrt{5}}$, we find that $\bar x=0$ is a solution
  (with $I=\emptyset$) as
  \begin{align*}
    A^T\frac{b-A\bar x}{\|A\bar x-b\|}=A^T\frac{b}{\|b\|}%
    =\begin{pmatrix}
       \lambda/2\\ \lambda \\ \lambda
     \end{pmatrix}
    \in \lambda \p \|\cdot\|_1(\bar x).
  \end{align*}
  We also see that $J=\{2,3\}$, so
  $A_J=\begin{pmatrix} 0 & 0 \\ 1 & 1 \end{pmatrix}$. Consequently,
  \cref{ass:intermediate} is violated at $\bar x$. In turn, let
  $\bar z:=\frac{\lambda}{6} \binom{\hphantom{-}2}{-1}\in \{b\}^\perp$. Then,
  with $A_{I^C}=A$, we find
  \begin{align*}
    \left\|A_{I^C}^T\left(\frac{b}{\|b\|}+\bar z\right)\right\|_\infty%
    = \left\|\begin{pmatrix}
               \lambda/2\\ \lambda \\ \lambda
             \end{pmatrix} %
    + \begin{pmatrix}
        \lambda/3\\- \lambda/6 \\ -\lambda/6
      \end{pmatrix}\right\|_\infty<\lambda.
  \end{align*}
  Therefore, in view of~\cref{eq:weak0}, \cref{ass:weak} holds at $\bar x$. \hfill $\diamond$
\end{example}
    
\begin{example} 
  Consider the \srlasso{}~\cref{eq:SR-LASSO} with
  \begin{align*}
    A
    :=\begin{pmatrix}
        1 & 0 & 2\\
        0 & 2 & -2 
      \end{pmatrix}\AND %
    b:=\begin{pmatrix}1\\ 1 \end{pmatrix}.
  \end{align*}
  For $\lambda=\sqrt{2}$, we find that $\bar x=0$ is a solution (with
  $I=\emptyset$) as
  \begin{align*}
    A^T\frac{b-A\bar x}{\|A\bar x-b\|}=A^T\frac{b}{\|b\|}%
    =\begin{pmatrix}
       \frac{1}{\lambda}\\ \lambda\\0
     \end{pmatrix}%
    \in \lambda \p\|\cdot\|_1(\bar x).
  \end{align*}
  In particular, $J=\{2\}$, thus $A_J=\binom{0}{2}$, hence
  \cref{ass:intermediate} is satisfied.  In turn, \cref{ass:Strong} is
  violated as $I \subsetneq J$.\hfill $\diamond$
\end{example}

\section{Lipschitz stability under the intermediate condition}
\label{sec:lipschitz-stability-solution-function}

In this section, we show that the intermediate condition~\cref{ass:intermediate}
yields directional differentiability and local Lipschitz continuity of the
solution function of the SR-LASSO~\cref{eq:SR-LASSO} at the point in question,
and we provide explicit Lipschitz bounds.  The key result is that,
under~\cref{ass:intermediate}, the subdifferential map of the objective function
of the \srlasso{} is \emph{strongly metrically
  regular}~\cite{dontchev2014implicit, rockafellar1998variational} at the point
in question, \ie{} it is invertible with locally Lipschitz inverse there. To
this end, given a positively homogenous map $H:\R^n\rightrightarrows \R^m$, its
\emph{outer norm} given by
\begin{align*}
  |H|^+:=\sup_{\|x\|\leq 1}\sup_{y\in H(x)}\|y\|.
\end{align*}

\begin{proposition}
  \label{prop:strong-metric-regularity--intermediate}
  Let $A \in \reals^{m\times n}$, $b \in \reals^{m}$ and let
  $\bar{x} \in \reals^{n}$ be a solution of \cref{eq:SR-LASSO} such that
  $b \neq A\bar{x}$. Define $T : \reals^{n} \rightrightarrows \reals^{n}$ by
  $
    T(x) %
    := \frac{1}{\lambda} \partial \left( \| A (\cdot) - b \| \right)(x) %
    + \partial \| \cdot \|_{1}(x)$.
  Under \cref{ass:intermediate},
  $\left| D^{*} T(\bar{x} \mid 0)^{-1} \right|^{+} < \infty$. Moreover, $T$ is
  strongly metrically regular at $(\bar{x}, 0)$.
\end{proposition}

\begin{proof}
  Define $\bar{r} := b - A \bar{x}$ and $\bar{y} := \bar{r} / \| \bar{r}
  \|$. Let $G := \nabla \| A (\cdot) - b \|$ and observe that $G$ has the form
  $G = \phi(A, b, 1, \cdot)$ for $\phi$ as defined in~\cref{lem:phi}.  By
  \cref{lem:Sum}(c), \cref{lem:phi}(d) and the symmetry of $DG(\bar x)$, one has
  \begin{align*}
    D^{*}T (\bar{x} \mid 0)(y) %
    & = \frac{1}{\lambda} DG(\bar{x})y + D^{*}\left( \partial \| \cdot \|_{1} \right)\left( \bar{x} \mid  \frac{1}{\lambda} A^{\top} \bar{y} \right)(y)
    \\
    & = \frac{1}{\lambda \| \bar{r} \|} A^{\top} \left( \identity - \bar{y} \bar{y}^{\top}\right) A y + D^{*} \left( \partial \| \cdot \|_{1} \right) \left( \bar{x} \mid  \frac{1}{\lambda} A^{\top} \bar{y} \right)(y).
  \end{align*}
  Thus, we have
  \begin{align*}
    & y \in D^{*} T(\bar{x} \mid 0)^{-1}(z)
    \\
    \iff & z \in D^{*} T \left( \bar{x} \mid 0 \right)(y)
    \\
    \iff & z - \frac{1}{\lambda \| \bar{r} \|} A^{\top} \left( \identity - \bar{y} \bar{y}^{\top} \right) Ay \in D^{*} \left( \partial \| \cdot \|_{1} \right) \left( \bar{x} \mid  \frac{1}{\lambda } A^{\top} \bar{y} \right)(y)
    \\
    \iff & \left( z - \frac{1}{\lambda \| \bar{r} \|} A^{\top} \left( \identity - \bar{y} \bar{y}^{\top} \right) Ay , -y \right) \in N_{\gph \partial \| \cdot \|_{1}} \left( \bar{x} ,\frac{1}{\lambda } A^{\top} \bar{y} \right)
    \\
    \implies
    & %
      \begin{cases}
        \left( z_{i} - \frac{1}{\lambda \| \bar{r} \|} A_{i}^{\top} \left( \identity - \bar{y} \bar{y}^{\top} \right) Ay , -y_{i} \right) \in \reals \times \{0\}, & \forall i \in J^{C},\\
        \left( z_{i} - \frac{1}{\lambda \| \bar{r} \|} A_{i}^{\top} \left( \identity - \bar{y} \bar{y}^{\top} \right) Ay , -y_{i} \right) \in \{0\} \times \reals, & \forall i \in I,\\
        \left( z_{i} - \frac{1}{\lambda \| \bar{r} \|} A_{i}^{\top} \left( \identity - \bar{y} \bar{y}^{\top} \right) Ay , -y_{i} \right) \in \reals_{-}\times \reals_{+} \cup \reals_{+}\times \reals_{-}, & \forall i \in J\setminus I.
    \end{cases}
  \end{align*}
  The third equivalence holds by definition of the coderivative, and the final
  implication can be obtained, for instance,
  from~\cite[Lemma~4.9]{berk2023lasso}. The first two conditions, for $J^{C}$
  and $I$, respectively, yield
  $ y_{J^{C}} %
    =0$ and $
    \lambda \| \bar{r} \| z_{I} %
    = A_{I}^{\top} \left( \identity - \bar{y}\bar{y}^{\top} \right) A_{J} y_{J}$.
  Notice that the third condition (using $y_{J^{C}} \equiv 0$) implies
  \begin{align*}
    y_{i} \left( z_{i} - \frac{1}{\lambda \| \bar{r} \|} A_{i}^{\top} \left( \identity - \bar{y} \bar{y}^{\top} \right) A_{J}y_{J} \right) \geq 0 \qquad \forall i \in J\setminus I.
  \end{align*}
  Combining this observation with the first two conditions yields 
  \begin{equation}
    \label{eq:Relax}
    y_{J}^{\top} z_{J} - \frac{1}{\lambda \| \bar{r} \|} y_{J}^{\top} A_{J}^{\top} (\identity - \bar{y} \bar{y}^{\top})A_{J} y_{J}\geq 0 \quad \forall z \in D^{*} T \left( \bar{x} \mid 0 \right)(y).
  \end{equation}
  Therefore, we find that
  \begin{align*}
    \left| D^{*} T(\bar{x} \mid 0)^{-1} \right|^{+} %
    & = \sup_{(z, y) \in \reals^{n} \times \reals^{n}} \left\{ \| y \| \mid \| z \| \leq 1 , y \in D^{*} T(\bar{x} \mid 0)^{-1}(z)\right\}\\
    & \leq \sup_{(z,u)\in \reals^{n} \times \reals^{|J|}} \set{\|u\|}{ \| z \| \leq 1,\;u^{\top} z_{J}\geq  \frac{1}{\lambda \| \bar{r} \|} u^{\top} A_{J}^{\top} (\identity - \bar{y} \bar{y}^{\top})A_{J} u}\\
    & < +\infty.
  \end{align*}
  Here the finiteness is due to the fact that the second supremum is attained by
  compactness of the constrained set which, in turn, relies on the positive
  definiteness of $A_{J}^{\top} (\identity - \bar{y} \bar{y}^{\top})A_{J} $ due
  to \cref{ass:intermediate} and \cref{lem:SM}. Hence,
  by~\cite[Theorem~4C.2]{dontchev2014implicit} it follows that $T$ is metrically
  regular at $(\bar{x}, 0)$. In addition, a subdifferential of a closed, proper
  convex function, $T$ is globally (maximally) monotone, so
  by~\cite[Theorem~3G.5]{dontchev2014implicit}, it follows that $T$ is strongly
  metrically regular at $(\bar{x}, 0)$.
\end{proof}

We use the strong metric regularity result under \cref{ass:intermediate} to
bootstrap our way to directional differentiability and obtain a (local)
Lipschitz modulus for the solution map that depends on $J$. For this, we need
the following preparatory result.

\begin{lemma}
  \label{lem:local}
  Let $\bar x$ be the (unique) solution of~\cref{eq:SR-LASSO} such that (given
  $(A, b, \lambda)$) \cref{ass:intermediate} holds at $\bar x$. Suppose that
  $\bar x_k$ solves~\cref{eq:SR-LASSO} given
  $(A_k,b_k,\lambda_k)\to (A,b,\lambda)$ and assume $\bar x_k\to \bar x$. Then,
  \cref{ass:intermediate} holds at $\bar x_k$ (for $(A_k,b_k,\lambda_k)$) for all $k\in\bN$ sufficiently large. 
\end{lemma}
\begin{proof}
  First, note that $A\bar x\neq b$. As $(A_k,b_k)\to (A,b)$ and
  $\bar x_k\to \bar x$, by continuity we find that $A_k\bar x_k\neq b_k$ for all
  $k$ sufficiently large. In particular, the equicorrelation set $J_k$
  associated to $\bar x_k$ and $(A_k,b_k,\lambda_k)$ is well-defined for such
  $k$, and by continuity, $J_k\subseteq J$ for all $k$ sufficiently large. Since
  $A_J$ has full rank, so does $A_{J_k}$ for all $k$ sufficiently large. It remains to observe that $b_k\notin A_{J_k}$ for all $k\in\bN$ sufficiently large. This,  however, can easily be seen (by contradiction),  using that $b_k\to \bar b\not \in \rge A_J$ and $\rge A_{J_k}\subset \rge A_J$ for  all $k\in\bN$ sufficiently large as argued above.
\end{proof}

\noindent
We are now in a position to state the main result of this section. Here, note that we call a (single-valued) map $F:X\subset\bE_1\to\bE_2$ {\em locally Lipschitz} at $\bar x\in X$ if there exist $L,\varepsilon>0$ 
(possibly depending on $\bar x$) such that
\[
\|F(x)-F(y)\|\leq L \|x-y\|\quad \forall x,y\in B_\varepsilon(\bar x).
\footnote{$B_\varepsilon(\bar x):= \bar x+\varepsilon \bB$ is the Euclidean ball with radius $\varepsilon>0$ centered at $\bar x$.}
\]
Recall that we already know from \cref{prop:intermed-implies-weak} and
\cref{th:sufficient-condition-for-uniqueness} that the intermediate condition in
\cref{ass:intermediate} implies uniqueness of solutions at the point in
question.

\begin{theorem}
  \label{thm:directional-differentiability}
  Let $(\bar b,\bar \lambda)\in \R^m\times \R_{++}$ and suppose that
  \cref{ass:intermediate} holds at $\bar{x} := S(\bar{b}, \bar{\lambda})$, where
  $S$ is defined as in~\cref{eq:solution-mapping-b-lambda}. Then:
  \begin{enumerate}[label=(\alph*)]
  \item $S$ is locally Lipschitz at $(\bar{b}, \bar{\lambda})$ with (local)
    Lipschitz modulus
    \begin{align*}
      L %
      \leq \left[ \frac{1}{\sigma_{\min}(A_{J})^{2}} + \frac{1}{1 - \|A_{J} A_{J}^{\dagger}\bar{y}\|} \right] %
      \cdot \left[ \sigma_{\max}\left( A_{J} \right) + \left\| \frac{A_{J}^{\top}(A\bar{x} - \bar{b})}{\bar{\lambda}} \right\| \right].
    \end{align*}
  \item $S$ is directionally differentiable at $(\bar{b}, \bar{\lambda})$ and
    the directional derivative\newline{}
    $S'((\bar{b}, \bar{\lambda}); (\cdot , \cdot)) : \reals^{m} \times \reals
    \to \reals^{n}$ is locally Lipschitz. Moreover, for
    $(q, \alpha) \in \reals^{m} \times \reals$ there exists
    $K = K(q, \alpha) \subseteq J$ with $\supp(\bar{x}) \subseteq K$ such that
    \begin{align*}
      S'((\bar{b}, \bar{\lambda}); (q, \alpha)) %
      = L_{K} \left( B \left( A_{K}^{\top} (\identity - \bar{y} \bar{y}^{\top}) q + \frac{\alpha}{\bar{\lambda}}A_{K}^{\top} (A S(\bar{b}, \bar{\lambda}) - \bar{b}) \right)\right),
    \end{align*}
    where
    $B:= \left( A_{K}^{\top} A_{K} \right)^{-1} + \dfrac{A_{K}^{\dagger} \bar{y}
      (A_{K}^{\dagger}\bar{y})^{\top}}{1 - \bar{y}^{\top} A_{K} A_{K}^{\dagger}
        \bar{y}}$.
  \end{enumerate}
\end{theorem}

\begin{proof}
  We apply~\cite[Proposition~4.10]{berk2023lasso} for 
  $f : (\reals^{m} \times \reals) \times \reals^{n} \to \reals^{n}$  with 
  \begin{align*}
    f((b, \lambda), x) %
    = \frac{1}{\lambda} A^{\top} \partial \|\cdot\|(Ax - b), %
    \quad\forall b \in \reals^m, \; \forall \lambda>0, \; \forall x\in\reals^n,
  \end{align*}
  and
  $F : \reals^{n} \rightrightarrows \reals^{n}, \;F := \partial \| \cdot
  \|_{1}$. Throughout, to simplify notation, we make the identification
  $f(b, \lambda, x) := f((b, \lambda), x)$ (and perform this unnesting
  elsewhere, where appropriate). Under~\cref{ass:intermediate}, it holds that
  $A\bar{x} \neq \bar{b}$, hence, $f$ is continuously differentiable in a
  neighborhood of $(\bar{b}, \bar{\lambda}, \bar{x})$. Additionally, $f$ and $F$
  are monotone, because the (sub)differential operator of a convex function is
  (maximally) monotone~\cite[Chapter~12]{rockafellar1998variational}. We
  organize the proof into three steps.

  \paragraph{Step 1. Local Lipschitz continuity of $S$} By construction,
  $(\bar{b}, \bar{\lambda}, \bar{x}) \in \gph S$ with
  $0 \in f(\bar{b}, \bar{\lambda}, \bar{x}) + F(\bar{x})$. For
  $T(x) := f(\bar{b}, \bar{\lambda}, x) + F(x)$, as
  in~\cref{prop:strong-metric-regularity--intermediate}, \cref{lem:Sum} yields
  \vspace{-6pt}
  \begin{align*}
    D^{*}T ( \bar{x} \mid 0) %
    = D_{x} f(\bar{b}, \bar{\lambda}, \bar{x})^{*} %
    + D^{*}F(\bar{x} \mid - f(\bar{b}, \bar{\lambda}, \bar{x})). 
  \end{align*}
  Hence, \cref{prop:strong-metric-regularity--intermediate} establishes
  finiteness of
  $\left| \left( D^{*}T ( \bar{x} \mid 0) \right)^{-1} \right|^{+}$, giving
  local Lipschitz continuity of $S$ at $(\bar b, \bar \lambda)$ 
  by~\cite[Proposition~4.10(b)]{berk2023lasso} and showing the validity of the first claim in part (a). 
  
  \paragraph{Step 2. Directional differentiability of $S$ at $(\bar b, \bar \lambda)$} Observe that
  \begin{align*}
    S(b, \lambda) = \{ x \in \reals^{n} : 0 \in G(b, \lambda, x)\}, \quad \text{with }
    G(b, \lambda, x) := f(b, \lambda, x) + F(x).
  \end{align*}
  Moreover, $F$ is \emph{proto-differentiable} at
  $(\bar{x}, -f(\bar{b}, \bar{\lambda}, \bar{x}))$
  by~\cite[Remark~1   and  Lemma~4]{friedlander2022perspective}. Hence,
  by~\cite[Proposition~4.10(c)]{berk2023lasso}, the graphical derivative
  $DS(\bar{b}, \bar{\lambda})$ is (single-valued and) locally Lipschitz with
  \begin{align*}
    DS(\bar{b}, \bar{\lambda})(q, \alpha) %
    = \left\{ %
    w \in \reals^{n} : 0 \in DG(\bar{b},\bar{\lambda}, \bar{x}\mid 0)(q,\alpha,w) %
    \right\}, %
    \qquad \forall (q,\alpha) \in \reals^{m} \times \reals.
  \end{align*}
  Using the graphical derivative sum rule in \cref{lem:Sum}(a) gives
  \begin{align*}
    DG(\bar{b}, \bar{\lambda}, \bar{x}\mid 0)(q,\alpha,w) = Df(\bar{b}, \bar{\lambda}, \bar{x})(q,\alpha,w) + DF(\bar{x}\mid -f(\bar{b}, \bar{\lambda}, \bar{x}))(w),
  \end{align*}
  where $Df(\bar{b}, \bar{\lambda}, \bar{x} \mid 0)(q,\alpha,w) %
  = D_{b}f(\bar{b}, \bar{\lambda}, \bar{x}) q + D_{\lambda}f(\bar{b},
  \bar{\lambda}, \bar{x}) \alpha + D_{x}f(\bar{b}, \bar{\lambda}, \bar{x})
  w$. Letting $\bar{r} := \bar{b} - A\bar{x}$ and
  $\bar{y} := \bar{r} / \| \bar{r} \|$, we use \cref{lem:phi} to compute:
\begin{align*}
  Df(\bar{b}, \bar{\lambda}, \bar{x})(q,\alpha,w) %
  & = - \frac{1}{\bar{\lambda} \| \bar{r} \|} \left[ \frac{A^{\top} - A^{\top} \bar{r} \bar{r}^{\top}}{\| \bar{r} \|^{2}} \right] q %
   + \frac{\alpha}{\bar{\lambda}^{2} \| \bar{r} \|} A^{\top} \bar{r} %
       + \frac{1}{\bar{\lambda} \| \bar{r} \|} \left[ A^{\top} A - A^{\top} \frac{\bar{r} \bar{r}^{\top}}{\| \bar{r} \|^2}  A \right]w
  \\
  & = - \frac{A^{\top}}{\bar{\lambda} \| \bar{r} \|} \left[ \left( \identity - \bar{y}\bar{y}^{\top} \right) \left( q - Aw \right) - \frac{\alpha }{\bar{\lambda}} \bar{r} \right].
\end{align*}
Altogether, we obtain that 
\begin{align*}
  0 %
  & \in DG(\bar{b}, \bar{\lambda}, \bar{x}\mid 0) (q,\alpha,w)
  \\
  & = - \frac{A^{\top}}{\bar{\lambda} \| \bar{r} \|} \left[ \left( \identity - \bar{y}\bar{y}^{\top} \right) \left( q - Aw \right) - \frac{\alpha }{\bar{\lambda}} \bar{r} \right] + D(\partial \| \cdot \|_{1})\left(\bar x \mid  A^{\top} \bar{y}\right)(w)
\end{align*}
is equivalent to
\begin{align*}
  \frac{A^{\top}}{\bar{\lambda} \| \bar r \|} \left[ \left( \identity - \bar{y}\bar{y}^{\top} \right) \left( q - Aw \right) - \frac{\alpha }{\bar{\lambda}} \bar{r} \right] \in D(\partial \| \cdot \|_{1})\left(\bar x \mid A^{\top} \bar{y}\right)(w).
\end{align*}
This, in turn, by the definition of the graphical derivative,  is equivalent to
\begin{align}
  \label{eq:graphical-derivative-membership-1}
  \left( w, \frac{A^{\top}}{\bar{\lambda} \| \bar{r} \|} \left[ \left( \identity - \bar{y}\bar{y}^{\top} \right) \left( q - Aw \right) - \frac{\alpha }{\bar{\lambda}} \bar{r} \right] \right) \in T_{\gph \partial \| \cdot \|_{1}} \left( \bar x, A^{\top} \bar{y} \right).
\end{align}
Let $I = \supp(\bar x)$ and recall~\cite[Lemma~4.9]{berk2023lasso}, namely,
\begin{align}
  \label{eq:TC-gph-d1norm-inclusion-1}
  T_{\gph \partial \|\cdot\|_{1}}(\bar x, \bar u) \subseteq \bigtimes_{i = 1}^{n}
  \begin{cases}
    \reals \times \{0\}, & \bar x_{i} \neq 0, \bar u_{i} = \sgn (\bar x_{i}), \\
    \reals_{-} \times \{0\} \cup \{0\} \times \reals_{+}, & \bar x_{i} = 0, \bar u_{i} = -1,\\
    \{0\} \times \reals_{-}  \cup \reals_{+} \times \{0\}, & \bar x_{i} = 0, \bar u_{i} = +1,\\
    \{0\} \times \reals, & \bar x_{i} = 0, |\bar{u}_{i}| < 1.
  \end{cases}
\end{align}
Using that $\| A_{J^{C}}^{\top} \bar{y} \|_{\infty} < \lambda$ and
$\| A_{J}^{\top} \bar{y} \|_{\infty} = \lambda$ with $I \subseteq J$ as well as
$\bar{x}_{J^{C}} = 0$, the inclusion~\cref{eq:TC-gph-d1norm-inclusion-1} and the
membership~\cref{eq:graphical-derivative-membership-1} together imply
\begin{align*}
  \begin{cases}
    \left( w_{i}, \frac{A_{i}^{\top}}{\bar{\lambda} \| \bar{r} \|} \left[ (\identity - \bar{y}\bar{y}^{\top})(q - Aw) - \frac{\alpha}{\bar{\lambda}}\bar{r} \right] \right) \in \reals \times \{0\}, & \forall i \in I, \\
    \left( w_{i}, \frac{A_{i}^{\top}}{\bar{\lambda} \| \bar{r} \|} \left[ (\identity - \bar{y}\bar{y}^{\top})(q - Aw) - \frac{\alpha}{\bar{\lambda}}\bar{r} \right] \right) \in \{0\} \times \reals, & \forall i \in J^{C}, \\
    \left( w_{i}, \frac{A_{i}^{\top}}{\bar{\lambda} \| \bar{r} \|} \left[ (\identity - \bar{y}\bar{y}^{\top})(q - Aw) - \frac{\alpha}{\bar{\lambda}}\bar{r} \right] \right) \in \{0\} \times \reals \cup \reals \times \{0\},  & \forall i \in J\setminus I.
  \end{cases}
\end{align*}
In particular, for any $(q, \alpha)$, $w_{J^{C}} = 0$. Thus, 
$A_{I}^{\top} \left[ (\identity - \bar{y}\bar{y}^{\top}) (q - A_{J} w_{J}) - \frac{\alpha}{\bar{\lambda}}\bar{r} \right] = 0$.
Likewise, for all $i \in J \setminus I$ we have
$
  w_{i} A_{i}^{\top} \left[ (\identity - \bar{y}\bar{y}^{\top})(q - A_{J} w_{J}) - \frac{\alpha}{\bar{\lambda}}\bar{r} \right] = 0$.
Now, set $K := I\cup\set{i \in J \setminus I}{w_{i} \neq 0}$ and note that $I \subseteq K \subseteq J$ and $w_{K^{C}} \equiv 0$. Consequently, 
$A_{K}^{\top} \left[ (\identity - \bar{y}\bar{y}^{\top})(q - A_{K} w_{K}) - \frac{\alpha}{\bar{\lambda}}\bar{r} \right] = 0$,
which is equivalent to
$
  A_{K}^{\top} (\identity - \bar{y}\bar{y}^{\top}) A_{K} w_{K} = A_{K}^{\top} (\identity - \bar{y}\bar{y}^{\top})q - \frac{\alpha}{\bar{\lambda}} A_{K}^{\top} \bar{r}$.
Note that $A_{K}$ has full column rank because $A_{J}$ does (by
\cref{ass:intermediate}). Using that $\bar{y} \notin \rge A_{K}$
because $\bar{b} \not\in \rge A_{J}$, \cref{lem:SM} yields
\begin{align*}
  B := \left[ A_{K}^{\top} (\identity - \bar{y}\bar{y}^{\top}) A_{K} \right]^{-1} = \left( A_{K}^{\top} A_{K}  \right)^{-1} + \frac{A_{K}^{\dagger} \bar{y} (A_{K}^{\dagger}\bar{y})^{\top}}{1 - \bar{y}^{\top} A_{K} A_{K}^{\dagger} \bar{y}}.
\end{align*}
In particular, using that $w_{K^{C}} \equiv 0$ (by definition of $K$), we see
that $w$ and $K$ are uniquely defined for a given $(q, \alpha)$ with
$w = DS(\bar{b}, \bar{\lambda})(q, \alpha)$, where
\begin{align}
  \label{eq:directional-differentiability-wK}
  w_{K} = B \left( A_{K}^{\top} (\identity - \bar{y}\bar{y}^{\top})q - \frac{\alpha}{\bar{\lambda}} A_{K}^{\top} \bar{r} \right), \qquad w_{K^{C}} \equiv 0. 
\end{align}
We conclude that $S$ is directionally differentiable at $(\bar{b}, \bar{\lambda})$ with directional derivative
$S'((\bar{b}, \bar{\lambda}); (q,\alpha) ) =  w$,
where $w = w(q, \alpha)$ is defined as in
\cref{eq:directional-differentiability-wK}. This proves part (b).

\paragraph{Step 3. Estimation of the Lipschitz modulus of $S$} To infer the
Lipschitz bound claimed in part (a) first note that, by \cref{lem:local}
combined with the fact that $S$ is (Lipschitz) continuous near $\bar x$, we can
infer that \cref{ass:intermediate} holds at every point $x=S(b,\lambda)$ for all
$(b,\lambda)$ sufficiently close to $(\bar b,\bar \lambda)$. Therefore, we can
reiterate the whole argument above to infer that $S$ is directionally
differentiable at $(b,\lambda)$ with the corresponding expression for the
directional derivative which is, in addition, (locally Lipschitz) continuous as
a function of the direction for all $(b,\lambda)$ sufficiently close to
$(\bar b,\bar \lambda)$. Hence, by~\cite[Proposition~4.10(c)]{berk2023lasso},
$S$ is locally Lipschitz at $(\bar{b}, \bar{\lambda})$ with modulus
\begin{align*}
  L %
  & := \limsup_{(b, \lambda) \to (\bar{b}, \bar{\lambda})} \max_{\| (q, \alpha) \| \leq 1} \| S'((b, \lambda);(q, \alpha)) \|. %
\end{align*}
Let $(b_{k}, \lambda_{k}) \to (\bar{b},\bar{\lambda})$ with
$ \max_{\| (q,\alpha) \| \leq 1} \| S' ((b_{k}, \lambda_{k}); (q, \alpha))\| \to
L.  $ As $S'((b_{k}, \lambda_{k}); (\cdot,\cdot))$ is continuous (as mentioned
above) for all $k \in \nats$ (sufficiently large), there exists
$(\bar{q}, \bar{\alpha}) \in \mathbb{B}$ and
$\{(q_{k}, \alpha_{k})\}_{k \in \nats} \subseteq \mathbb{B}$ with
$(q_{k}, \alpha_{k}) \to (\bar{q}, \bar{\alpha})$ such that
$\| S'((b_{k}, \lambda_{k}); (q_{k}, \alpha_{k})) \| \to L$.  Let the associated
index sets be $K_{k}$. By finiteness, we may assume without loss of generality
that $K_k\equiv K\subseteq J$. Thus, we have
\begin{align*}
  \left\| S'((b_{k}, \lambda_{k}); (q_{k}, \alpha_{k})) \right\| %
  & = \left\| L_{K} \left( B_{k} \left( A_{K}^{\top} (\identity - \bar{y}_{k}\bar{y}_{k}^{\top}) q_{k} - \frac{\alpha_{k}}{\lambda_{k}} A_{K}^{\top} r_{k} \right) \right) \right\| %
  \\
  & \leq \lambda_{\max}(B_{k}) \cdot \left\| A_{K}^{\top} (\identity - \bar{y}_{k}\bar{y}_{k}^{\top}) q_{k} - \frac{\alpha_{k}}{\lambda_{k}} A_{K}^{\top} r_{k} \right\|, %
\end{align*}
using that $\| L_{K} \| \leq 1$ and where
$r_{k} :=  b_{k} - A S(b_{k}, \lambda_{k})$,
$\bar{y}_{k} := r_{k} / \| r_{k} \|$, $B_{k} := \left[ A_{K}^{\top} (\identity - \bar{y}_{k}\bar{y}_{k}^{\top}) A_{K} \right]^{-1}$.
Here, observe that $B_k$ is well-defined as $\rge A_K\subseteq \rge A_J$ and $\bar y\notin A_J$, thus $\bar y_k\notin \rge A_K$ (for all $k$ sufficiently large). 
Passing to the limit yields
\begin{align*}
  L %
  & \leq \lambda_{\max}(B) \cdot \left\| A_{K}^{\top} (\identity - \bar{y}\bar{y}^{\top}) \bar{q} - \frac{\bar{\alpha}}{\bar{\lambda}} A_{K}^{\top} \bar r \right\| %
  \\
  & \leq \left[ \frac{1}{\sigma_{\min}(A_{K})^{2}} + \frac{\| A_{K}^{\top} \bar{y} \|^{2}}{1 - \bar{y}^{\top} A_{K} A_{K}^{\dagger}\bar{y}} \right] \cdot \left[ \sigma_{\max}\left( A_{K}^{\top} (\identity - \bar{y}\bar{y}^{\top}) \right) + \left\| \frac{A_{K}^{\top}(A\bar{x} - \bar{b})}{\bar{\lambda}} \right\| \right]
  \\
  & \leq \left[ \frac{1}{\sigma_{\min}(A_{K})^{2}} + \frac{1}{1 - \|A_{K} A_{K}^{\dagger}\bar{y}\|} \right] \cdot \left[ \sigma_{\max}\left( A_{K} \right) + \left\| \frac{A_{K}^{\top}(A\bar{x} - \bar{b})}{\bar{\lambda}} \right\| \right]
  \\
  & \leq \left[ \frac{1}{\sigma_{\min}(A_{J})^{2}} + \frac{1}{1 - \|A_{J} A_{J}^{\dagger}\bar{y}\|} \right] \cdot \left[ \sigma_{\max}\left( A_{J} \right) + \left\| \frac{A_{J}^{\top}(A\bar{x} - \bar{b})}{\bar{\lambda}} \right\| \right].
\end{align*}
Here, the second inequality uses \cref{lem:SM}(b) for the first factor, and that $\|(\bar q,\bar\alpha)\|\leq  1$ for the second.  
The penultimate inequality uses that $\|\bar y\|=1$, hence $\identity-\bar y\bar y^T$ is a projection, and thus
$\sigma_{\max}(A^T_K(\identity-\bar y\bar y^T))\leq \|A_K\|\cdot\|\identity-\bar y\bar y^T\|\leq \|A_K\|$. The last inequality uses that $\sigma_{\min}(A_J)\leq \sigma_{\min}(A_K)$ and $\sigma_{\max}(A_J)\geq  \sigma_{\max}(A_K)$; note  $\|A_{K} A_{K}^{\dagger}\bar{y}\|\leq \|A_{J} A_{J}^{\dagger}\bar{y}\|$: projecting onto a larger subspace does not decrease norm.
\end{proof}

\begin{remark} 
\label{rmk:directional-diff-neighborhood}
An inspection of the proof of \cref{thm:directional-differentiability} reveals
that the claim of \cref{thm:directional-differentiability}(b) can be
strengthened: the argument used at the beginning of Step 3 of the proof shows
that $S$ is directionally differentiable (with locally Lipschitz directional
derivative) not only at $(\bar b, \bar \lambda)$, but in a whole neighborhood.\hfill$\diamond$
\end{remark}

\section{Continuous differentiability of the solution function}
\label{sec:continuous-differentiability-solution-function}

In this section, we show that, under~\cref{ass:Strong}, the solution map is
continuously differentiable in a neighborhood of the (unique) solution. This is essentially a direct corollary of the directional differentiability result from \cref{thm:directional-differentiability} (b) once we establish that the support of solutions is locally constant. To this end, recall from~\cite[(2.4)]{berk2023lasso} that, for a (closed) proper, convex function $f:\R^n\to\rp$, one has
\begin{align}
  \label{eq:ri-normal-cone-connection}
  y \in \ri (\partial f(\bar{x})) \qquad\iff \qquad (y, -1) \in \ri N_{\epi f} (\bar x, f(\bar{x})).
\end{align}

\begin{lemma}[Constancy of support]
  \label{lem:constancy-of-support}
  For
  $(\bar{A}, \bar{b}, \bar{\lambda}) \in \reals^{m\times n} \times \reals^{m}
  \times \reals_{++}$ let $\bar{x}$ be the unique minimizer of~\cref{eq:SR-LASSO} such
  that \cref{ass:Strong}(ii) holds. Assume that
  $(A_{k}, b_{k}, \lambda_{k}) \to (\bar{A}, \bar{b}, \bar{\lambda})$ and that
  $x_{k}$ is a solution of~\cref{eq:SR-LASSO} given
  $(A_{k}, b_{k}, \lambda_{k})$ such that $x_{k} \to \bar x$. Then
  $\supp(x_{k}) = \supp (\bar x)$ for all $k$ sufficiently large.
\end{lemma}
\begin{proof} Set $\bar{z} := (\bar{x}, \| \bar{x} \|_{1})$,
  $\Omega := \epi \| \cdot \|_{1}$ and
  $\phi(x,t) := \frac{1}{\bar{\lambda}} \| \bar{A} x - \bar{b}\| + t$.  The
  proof follows the same reasoning as the proof
  of~\cite[Lemma~4.7]{berk2023lasso}, and all that needs to be observed is the
  fact that, by \cref{ass:Strong}(ii), and~\cref{eq:ri-normal-cone-connection},
  $\bar{z}$ is \emph{nondegenerate}, \ie{}
  $- \nabla \phi(\bar{z}) \in \ri N_{\Omega}(\bar{z})$.
\end{proof}

\noindent
We record the fact that \cref{ass:Strong} is a local property.

\begin{remark}[\cref{ass:Strong} is a local property]
  \label{rem:StrongLocal}
  Assume that \cref{ass:Strong} holds at $\bar x=S(\bar b, \bar \lambda)$. Since
  $S$ is (locally Lipschitz) continuous around $(\bar b,\bar \lambda)$,
  \cref{lem:constancy-of-support} yields a neighborhood $\cV$ of
  $(\bar b,\bar \lambda)$ such that
  $\supp(S(b,\lambda))\equiv \supp(S(\bar b,\bar\lambda))$ and, consequently,
  \cref{ass:Strong} holds at $S(b,\lambda)$ for all
  $(b,\lambda)\in \cV$.\hfill$\diamond$
\end{remark}

\begin{theorem}
  \label{thm:main-theorem}
  For $(\bar{b}, \bar{\lambda}) \in \reals^{m} \times \reals_{++}$ let $\bar{x}$
  be a solution of~\cref{eq:SR-LASSO} with $I := \supp(\bar{x})$ such
  that~\Cref{ass:Strong} holds. Then $S$, defined as in~\cref{eq:solution-mapping-b-lambda},
  is continuously differentiable at $(\bar{b}, \bar{\lambda})$ with derivative
  \begin{align*}
    DS(\bar{b}, \bar{\lambda})(q, \alpha) = L_{I} \left( \left[ \left( A_{I}^{\top} A_{I} \right)^{-1} + \frac{A_{I}^{\dagger} \bar{y} (A_{I}^{\dagger} \bar{y})^{\top}}{1 - \bar{y}^{\top} A_{I} A_{I}^{\dagger} \bar{y}} \right] \cdot \left[ A_{I}^{\top} (\identity - \bar{y}\bar{y}^{\top})q - \frac{\alpha}{\bar{\lambda}} A_{I}^{\top} \bar r \right] \right),
  \end{align*}
 for  $\bar r := \bar{b} - A\bar{x}$, $\bar{y} := \bar r/ \| \bar r \|$. In particular, $S$ is
  locally Lipschitz at $(\bar{b}, \bar{\lambda})$ with constant
  \begin{align*}
    L \leq \left[ \frac{1}{\sigma_{\min}(A_{I})^{2}} + \frac{1}{1 - \|A_{I} A_{I}^{\dagger}\bar{y}\|} \right] \cdot \left[ \sigma_{\max}\left( A_{I} \right) + \left\| \frac{A_{I}^{\top}\bar r}{\bar{\lambda}} \right\| \right].
  \end{align*}
\end{theorem}
\begin{proof} Recalling that \cref{ass:Strong} implies \cref{ass:intermediate} (due to $I=J$), we can already infer the Lipschitz bound from \cref{thm:directional-differentiability}. In addition, we can revisit the proof of the directional differentiability of $S$ under this premise to infer 
\[
S'((\bar b,\bar \lambda);(q,\alpha))=L_{I} \left( \left[ \left( A_{I}^{\top} A_{I} \right)^{-1} + \frac{A_{I}^{\dagger} \bar{y} (A_{I}^{\dagger} \bar{y})^{\top}}{1 - \bar{y}^{\top} A_{I} A_{I}^{\dagger} \bar{y}} \right] \cdot \left[ A_{I}^{\top} (\identity - \bar{y}\bar{y}^{\top})q - \frac{\alpha}{\bar{\lambda}} A_{I}^{\top} \bar r \right] \right).
\]
This directional derivative is linear in the direction $(q,\alpha)$, because $I=K(q,\alpha)=J$ here, and thus (see \cite[Proposition 4.10(c)]{berk2021lasso}) $S$ is, in fact, differentiable at $(\bar b,\bar\lam)$. Now, \cref{rem:StrongLocal} yields a neighborhood $\cV$ of $(\bar b,\bar \lam)$  such that 
\cref{ass:Strong} holds at $S(b,\lambda)$  with $\supp(S(b,\lambda))= I$ for all $(b,\lambda)\in \cV$. Therefore, reiterating the above argument, $S$ is differentiable at $(b,\lambda)\in \cV$ with the respective derivative which, by the constancy of the support, can be seen to be continuous. This proves continuous differentiability. 
\end{proof}

\begin{remark}
  \label{rmk:bk-not-in-Rge-A_I}
  In practice, it is reasonable to expect $\bar{b} \not\in \rge A_{I}$ in cases
  of interest to compressed sensing. There, it is interesting to consider
  $|I| \ll m$. Under mild assumptions, if $\bar{b}$ has been corrupted by random
  noise then it will be ``full dimensional'', in the sense of not being contained in
  any of the possible subspaces $\rge A_{I}$. \hfill$\diamond$
  \end{remark}

\begin{corollary}
\label{cor:Lipschitz_lambda}
  For $(\bar{b}, \bar{\lambda}) \in \reals^{m} \times \reals_{++}$ let $\bar{x}$
  be a solution of~\cref{eq:SR-LASSO} such
  that~\Cref{ass:Strong} holds and let $I := \supp(\bar{x})$. Then
  \begin{align*}
    S : \lambda \in \reals_{++} \mapsto \argmin_{x \in \reals^{n}} \left\{ \| Ax - \bar{b} \| + \lambda \| x \|_{1} \right\}
  \end{align*}
  is continuously differentiable at $\bar{\lambda}$ with derivative
  \begin{align*}
    DS(\bar{\lambda})(\alpha) = \frac{\alpha}{\bar{\lambda}} \left[ A_{I}^{\dagger} \bar{r} + \frac{A_{I}^{\dagger} \bar{y} (A_{I}^{\dagger} \bar{y})^{\top} A_{I}^{\top} \bar{r}}{1 - \bar{y}^{\top} A_{I} A_{I}^{\dagger} \bar{y}} \right], \quad \forall \alpha \in \mathbb{R},
  \end{align*}
  where $\bar{r} := \bar{b} - A\bar{x}$, $\bar{y} := \bar{r}/ \| \bar{r} \|$. In particular, $S$ is
  locally Lipschitz at $\bar{\lambda}$ with constant
  \begin{align}
    \label{eq:lipschitz_ub_lamda}
    L \leq \frac{1}{\bar{\lambda}} \left\| A_{I}^{\dagger} \bar{r} \right\| \cdot \left| 1-V \right|^{-1} %
    \leq \frac{\| A \bar{x} - \bar{b} \|}{\bar{\lambda} \cdot \sigma_{\min}(A_{I}) \cdot |1 - V|}, \quad V := \bar{y}^{\top} A_{I}A_{I}^{\dagger} \bar{y}.
  \end{align}
\end{corollary}

\begin{proof}
  In the proof of~\cref{thm:main-theorem}, the expression for the derivative,
  when $S$ is a function of $\lambda$ only, clearly reduces to
  $ S'(\bar{\lambda}; \alpha) = DS(\bar{\lambda})(\alpha) %
  = L_{I} \left( \frac{\alpha}{\bar{\lambda}} B A_{I}^{\top} \bar{r} \right)$
  where $B$ is defined as in \Cref{thm:directional-differentiability}(b).
  Accordingly, recalling that $V = \bar{y}^{\top} A_{I}A_{I}^{\dagger} \bar{y}$,
  \begin{align*}
    B A_{I}^{\top} \bar{r}
     = \left[ \left( A_{I}^{\top} A_{I} \right)^{-1} + \frac{A_{I}^{\dagger} \bar{y} (A_{I}^{\dagger} \bar{y})^{\top}}{1 - V} \right] A_{I}^{\top} \bar{r} %
     = A_{I}^{\dagger} \bar{r} + \frac{A_{I}^{\dagger} \bar{y} (A_{I}^{\dagger} \bar{y})^{\top}A_{I}^{\top} \bar{r}}{1 - V}
     = A_{I}^{\dagger} \bar{r} + \frac{A_{I}^{\dagger} \bar{r} V}{1 - V}.
  \end{align*}
  In particular,
  \begin{align*}
    L \leq \frac{1}{\bar{\lambda}}\left\| L_{I}\left( BA_{I}^{T} \bar{r} \right) \right\| %
       \leq \frac{1}{\bar{\lambda}} \left\| A_{I}^{\dagger} \bar{r} \right\| \cdot \left| 1 + \frac{V}{1-V} \right| %
        = \frac{1}{\bar{\lambda}} \left\| A_{I}^{\dagger} \bar{r} \right\| \cdot \left| 1-V \right|^{-1}. %
  \end{align*}
\end{proof}

\section{SR-LASSO \vs{} LASSO}
\label{sec:sr-vs-uc}

In this section we compare the Lipschitz behavior for the \srlasso{} solution map with that of the (unconstrained) \lasso{}. We draw theoretical comparisons (in \cref{sec:theoretical-comparison}) using the Lipschitz bound in~\cref{cor:Lipschitz_lambda} and an analogous bound derived in~\cite{berk2023lasso}; and numerical comparisons in~\cref{sec:numerical-lipschitz-comparison}.

\subsection{Comparison of Lipschitz bounds}
\label{sec:theoretical-comparison}

Let us recall that, for given $(A,b,\lambda) \in \mathbb{R}^{m \times n} \times \mathbb{R}^m \times \mathbb{R}_{++}$, the (unconstrained) \lasso{} is given by
\begin{align}
  \label{eq:unconstrained-LASSO}
  \min_{z \in \reals^{n}} \frac{1}{2} \| Az - b \|^{2} + \lambda \| z \|_{1}.
\end{align}
As mentioned in the introduction, the vital difference with respect to the \srlasso{} is the square on the data fidelity term. A variational analysis of its solution map is carried out in~\cite{berk2023lasso}. Here, we want to compare Lipschitz bounds for SR-LASSO and LASSO from a theoretical viewpoint. For the sake of simplicity, we focus on the regularity of the solution map with respect to the tuning parameter $\lambda$, although a similar comparison can be  made when the solution map is considered as a function of $(b,\lambda)$. To denote Lipschitz constants associated with \srlasso{} and (unconstrained) \lasso{}, we shall use the subscripts SR and UC, respectively.

\Cref{cor:Lipschitz_lambda} states that, under  \Cref{ass:Strong} with associated $(\bar b, \bar \lambda) \in \mathbb{R}^m \times \mathbb{R}_{++}$, a Lipschitz bound for the \srlasso{} solution map at the point $\bar \lambda$ (corresponding to the unique solution $\bar x_{\mathrm{SR}}$) is
\begin{equation}
\label{eq:Lipschitz_lambda_SR-LASSO}
L_{\mathrm{SR}} \leq \frac{1}{\bar \lambda}\|A_{I_\mathrm{SR}}^\dagger \bar r_{\mathrm{SR}}\|\cdot\left|1-\bar{y}^\top A_{I_\mathrm{SR}} A_{I_\mathrm{SR}}^\dagger \bar{y}\right|^{-1},
\end{equation}
where $I_\mathrm{SR} := \supp(\bar x_\mathrm{SR})$, $\bar{y} := \bar r_{\mathrm{SR}}/\|\bar r_{\mathrm{SR}}\|$ and $\bar{r}_{\mathrm{SR}} := \bar{b} - A\bar{x}_{\mathrm{SR}}$.

An analogous version of this bound for the \lasso{} can be derived from results in \cite{berk2023lasso}. Under \cite[Assumption~4.4]{berk2023lasso} with associated $(\bar b, \bar \lambda) \in \mathbb{R}^m \times \mathbb{R}_{++}$ (which is the analogue of \Cref{ass:Strong} for the \lasso{} case), \ie{} for $\bar x_{\mathrm{UC}}$ solving \cref{eq:unconstrained-LASSO} with $(b,\lambda) = (\bar b, \bar \lambda)$ such that
$$
A_I \text{ has full column rank} \quad \text{and} \quad \left\|A^\top_{I^C}(\bar b - A \bar x_{\text{UC}})\right\|_{\infty} < \lambda,
$$
where $I_{\text{UC}} := \supp(\bar x_{\text{UC}})$. In this case, an inspection of the proof of \cite[Corollary~4.16]{berk2023lasso} reveals that the derivative of the \lasso{} solution map $S_{\mathrm{UC}}$ at $\bar \lambda$ satisfies 
$
\| S_{\mathrm{UC}}'(\bar\lambda)\|
\leq \left\|\frac{1}{\bar \lambda}A_{I_{\mathrm{UC}}}^\dagger \bar r_{\mathrm{UC}}\right\| 
$,
where $\bar r_{\mathrm{UC}} := \bar b - A \bar x_{\mathrm{UC}}$, $\bar x_{\mathrm{UC}}$ is the unique LASSO solution, and $I_{\mathrm{UC}} := \supp(\bar x_{\mathrm{UC}})$. This leads, in turn, to the following Lipschitz bound: 
\begin{equation}
\label{eq:Lipschitz_lambda_LASSO}
L_{\mathrm{UC}} \leq \frac{1}{\bar \lambda}\|A_{I_{\mathrm{UC}}}^\dagger \bar r_{\mathrm{UC}}\|.
\end{equation}
We are now in a position to compare the two Lipschitz bounds \cref{eq:Lipschitz_lambda_SR-LASSO} and \cref{eq:Lipschitz_lambda_LASSO}. Under the respective ``strong'' assumptions (\ie{} \Cref{ass:Strong} and \cite[Assumption~4.4]{berk2023lasso}) and supposing that $\bar x_{\mathrm{SR}} \approx \bar{x}_{\mathrm{UC}}$ (which, in turn, implies $\bar r_{\mathrm{SR}} \approx \bar r_{\mathrm{UC}}$) and $I_{\mathrm{SR}} \approx I_{\mathrm{UC}}$, the only difference between the two bounds is the multiplicative term
$|1-\bar{y}^\top A_{I_\mathrm{SR}} A_{I_\mathrm{SR}}^\dagger \bar{y}|^{-1}$
present in the \srlasso{} case. Since $\|\bar{y}\|=1$ and recalling that $A_{I_\mathrm{SR}} A_{I_\mathrm{SR}}^\dagger$ is an orthogonal projection onto a subspace, we have $
0 < \bar{y}^\top A_{I_\mathrm{SR}} A_{I_\mathrm{SR}}^\dagger \bar{y}
\leq  \|\bar{y}\| \|A_{I_\mathrm{SR}} A_{I_\mathrm{SR}}^\dagger \bar{y}\| \leq \|\bar{y}\|^2 =1$.
This implies that $|1-\bar{y}^\top A_{I_\mathrm{SR}} A_{I_\mathrm{SR}}^\dagger \bar{y}|^{-1} > 1$, which shows that the Lipschitz bound for \srlasso{} is strictly larger than the one for the \lasso{}. 

Since we are comparing \emph{upper bounds}, strictly speaking we cannot conclude that the \emph{actual} Lipschitz constant of the \srlasso{} is larger than that of the \lasso{}. However, the fact that these two upper bounds arise from the application of analogous proof techniques suggests this to be a reasonable conjecture. This theoretical insight aligns with numerical evidence provided by \cref{fig:motivation-new} and the next subsection.

\subsection{Numerical Lipschitz comparison}
\label{sec:numerical-lipschitz-comparison}

Here, we numerically examine the solution sensitivity of \srlasso{} and
\lasso{}, complementing the theoretical observations of the previous
subsection. After providing implementation details in
\cref{sec:implementation-details}, we provide extended discussion and details
on~\cref{fig:motivation-sr-uc-comparison} in
\cref{sec:a-motivating-numerical-example} and then compare the two programs
in~\cref{sec:gamma-m-solution-sensitivity} in the context of varying measurement
size and noise scale.

\subsubsection{Implementation details}
\label{sec:implementation-details}

All numerical aspects, including solvers for \cref{eq:SR-LASSO}
and~\cref{eq:unconstrained-LASSO} were implemented in Python using \cvxpy{}
v1.2~\cite{agrawal2018rewriting, diamond2016cvxpy} with the \mosek{}
solver~\cite{mosek}. Default parameter settings were used. Some code and
extended discussion for the experiments in this work are available in
our~\href{https://github.com/asberk/srlasso_revolutions}{code
  repository}~\cite{coderepo}.

The elements of the experimental setup are as follows: the ground-truth signal
$x^\sharp \in \reals^{n}$ and measurements $b \in \reals^{m}$ are
given by
\begin{align}
\label{eq:CS_model}
  x^\sharp_{j} :=
  \begin{cases}
    m + W_{j}\sqrt m, & j \in [s],
    \\
    0 & j \in [n]\setminus [s]
  \end{cases}
  \AND b := A x^\sharp + \gamma w,
\end{align}
where $W_{j} \iid \mathcal{N}(0, 1)$, $A_{ij} \iid \mathcal{N}(0, m^{-1})$, and
$w_{i} \iid \mathcal{N}(0, 1)$ are all mutually independent. Here
$\mathcal{N}(\mu, \sigma^{2})$ denotes the normal distribution with mean
$\mu \in \reals$ and variance $\sigma^{2} > 0$ and, \eg{}
$w_{i} \iid \mathcal{N}(0, 1)$ means that the random vector $w$ has entries that
are indpendent identically distributed standard normal random variables. Above,
$s$ and $m$ are positive integers. In~\cref{fig:motivation-known} we set
$(m, n, s) = (50, 100, 5)$ and vary the noise scale $\gamma$;
in~\cref{fig:motivation-new,fig:motivation-new-continued} we set $(m, n, s, \gamma) = (100, 200, 5,
0.5)$. In~\cref{sec:gamma-m-solution-sensitivity} we fix the sparsity to $s = 7$
and vary the noise scale $\gamma$ and measurement size $m$.

Recall that we use SR and UC to refer to \srlasso{} and unconstrained \lasso{},
respectively. For $P \in \{\text{SR}, \text{UC}\}$ and $\Lambda \in 2\nats + 1$,
suppose that $(\lambda_{i}^{P})_{i \in [\Lambda]}$ is a grid of values for the
regularization parameter, logarithmically spaced about asymptotically
order-optimal parameter choices (\eg{} see~\cite[(6)]{belloni2011square}
and~\cite[Theorem~6.1]{bickel2009simultaneous}, respectively)
\begin{align*}
  \lambda_{*}^{\text{SR}} &:= 1.1 \cdot \Phi^{-1} \left( 1 - \frac{0.05}{2 n} \right) %
  & %
    \lambda_{*}^{\text{UC}} &:= \sqrt{2 \log(n)},
\end{align*}
where $\Phi$ is the cdf of the normal distribution (refer to
\code{numpy.logspace} for details of the grid
generation~\cite{harris2020array}). Define
$\bar{x}^{P}(\lambda) = \bar{x}^{P}(A, b, \lambda) \in \reals^{n}$ to be a
solution to $P$ for given parameters $(A, b, \lambda)$. We use this notation to
refer to the numerical solutions computed in Python throughout our experiments,
and may safely overlook any issues related to non-uniqueness. For
$P \in \{ \text{SR}, \text{UC}\}$, define
\begin{align*}
  \bar \lambda_{\text{best}}^{P} &:= \argmin_{\lambda \in ( \lambda_{i}^{P})_{i \in [\Lambda]}} \|\bar{x}^{P}(\lambda) - x^{\sharp}\|,  %
  & %
    \bar{x}^{P}_{\text{best}} &:= \bar{x}^{P}(\bar \lambda_{\text{best}}^{P}), %
\end{align*}
and define the normalized parameters $\lambda_{\text{nmz}}$ by
$\lambda_{\text{nmz}}^{P} := \left( \lambda_{i}^{P} /
  \bar{\lambda}_{\text{best}}^{P} \right)_{i \in [\Lambda]}$. If we could be
referring to either program or if clear from context, then we may omit the
superscript. For example, we may simply refer to $\bar \lambda_{\text{best}}$,
rather than $\bar \lambda_{\text{best}}^{P}$, where $\bar \lambda_{\text{best}}$
could correspond to either program $P \in \{\text{SR}, \text{UC}\}$. Finally, we
refer to the quantity
$\|\bar{x}(\lambda) - \bar{x}(\bar{\lambda})\| / \|\bar{x}(\bar{\lambda})\|$ as
\emph{relative error} (viewed as a function of $\lambda$, with a fixed reference
value $\bar{\lambda}$).

\subsubsection{Robustness-sensitivity trade off for parameter tuning}
\label{sec:a-motivating-numerical-example}

In~\cref{fig:motivation-sr-uc-comparison} we presented two graphics that serve
to orient and motivate our work, in particular suggesting that there is a trade
off for parameter tuning between robustness and
sensitivity. In~\cref{fig:motivation-known} we demonstrated graphically the
known fact that $\lambda_{\text{best}}^{\text{UC}}$ depends on the noise scale
$\gamma > 0$, whereas $\lambda_{\text{best}}^{\text{SR}}$ is relatively robust
(\ie{} agnostic) to variation in the noise scale. Five independent trials were
performed for each program using parameter values $(m, n, s) = (50, 100,
5)$. Aspects of the experimental setup not already detailed in
\cref{sec:contributions}, including the definition of the sensing matrix
$A \in \reals^{m \times n}$, measurements $b \in \reals^{m}$ and ground-truth
signal $x^{\sharp}$, are detailed in~\cref{sec:implementation-details}.

\begin{figure}
    \centering
    \includegraphics[width=.7\linewidth]{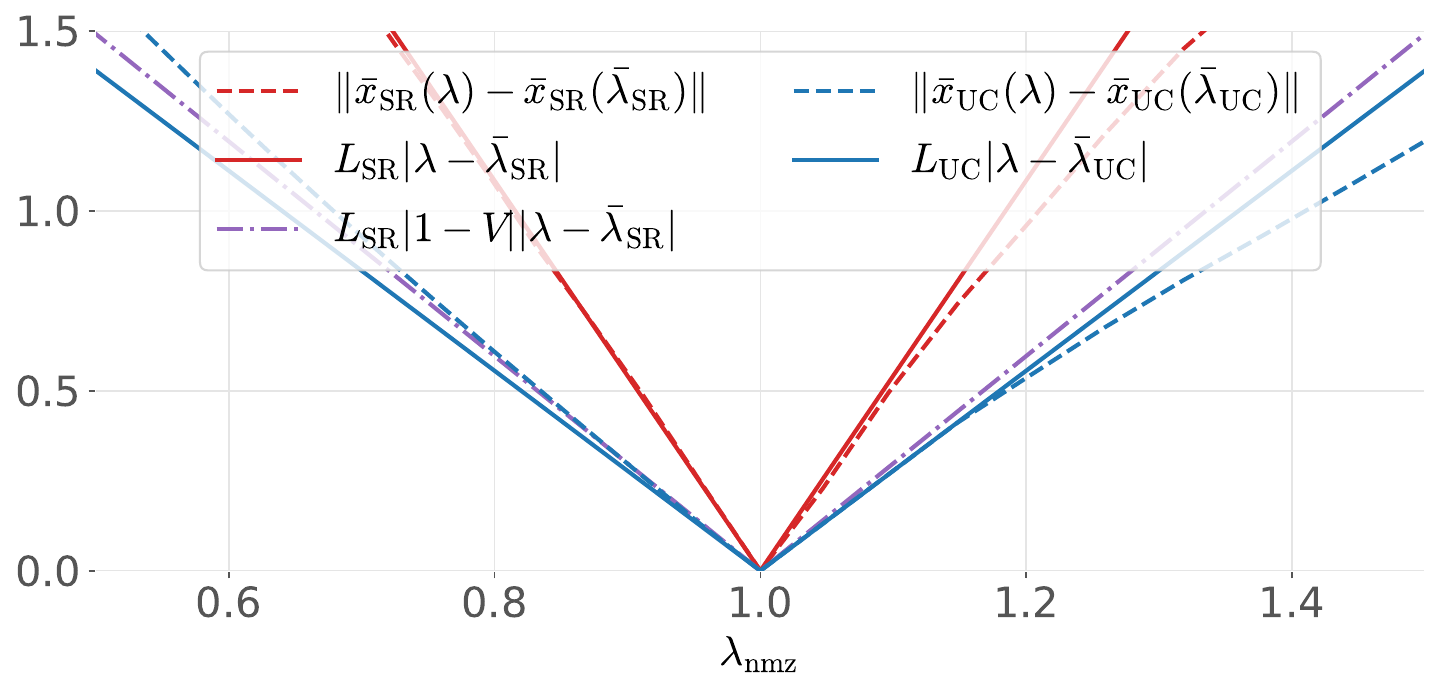}
    \caption{(Local) Lipschitz behavior for each program; $L_{\text{SR}}$ as
      in~\cref{eq:lipschitz_ub_lamda}, $L_{\text{UC}}$ as
      in~\cref{eq:Lipschitz_lambda_LASSO}. $V$ as in~\cref{cor:Lipschitz_lambda}
      gives
      $L_{\text{UC}} \approx L_{\text{SR}} |1-V|$. Note $\lambda_{\text{nmz}}$ is defined in \cref{sec:implementation-details}.\label{fig:motivation-new-continued}}
\end{figure}
On the other hand, in~\cref{fig:motivation-new-continued} we plot the empirical local
Lipschitz behavior of the \srlasso{} and \lasso{} solution maps --- namely,
$\| \bar{x}(\lambda) - \bar{x}(\bar{\lambda}) \|$. Here,
$\bar{\lambda} \approx \lambda_{\text{best}}$ and
$\lambda_{\text{nmz}} := \lambda / \bar{\lambda}$ so that the two programs can
be plotted about the same reference point on the horizontal
axis. $L_{\text{SR}}$, given in \cref{eq:lipschitz_ub_lamda}, corresponds to a
theoretical upper bound on the local Lipschitz constant for
$\bar{x}_{\text{SR}}(\lambda)$, established in \cref{cor:Lipschitz_lambda};
$L_{\text{UC}}$ to that for \lasso{}, obtained by the current authors in a
previous work (in particular, a tighter version
of~\cite[Theorem~4.13]{berk2023lasso}). Interestingly, there is a clear
connection to be drawn between the pair $L_{\text{SR}}$ and $L_{\text{UC}}$,
which is made precise in \cref{sec:theoretical-comparison}. There, and in
\cref{cor:Lipschitz_lambda} we define a quantity $V$ (appearing in the legend of
\cref{fig:motivation-new-continued}), which appears to serve as a good characterization
for how the two local Lipschitz constants differ, namely
$L_{\text{UC}} \approx L_{\text{SR}} |1 - V|$. Note that we chose
$\bar{\lambda}$ to correspond with a good estimate of the ground truth signal
$x^{\sharp}$, because this is perhaps the most interesting region of parameter
space in practice; however, the observations made here apply to a much larger
range of $\lambda$ values. The dimensional parameters for this plot are identical to those for \cref{fig:motivation-known,fig:motivation-new} and are given in \cref{sec:implementation-details}.

\subsubsection{Effect of noise scale and measurement size}
\label{sec:gamma-m-solution-sensitivity}


We next examine empirically the effect of noise scale and measurement size
$(\gamma, m)$ on the solution sensitivity between~\cref{eq:SR-LASSO}
and~\cref{eq:unconstrained-LASSO} in~\cref{fig:noisescale-experiment}.
\begin{figure}[ht!]
  \centering
  \begin{subfigure}{0.95\linewidth}
    \centering
    \includegraphics[width=\linewidth]{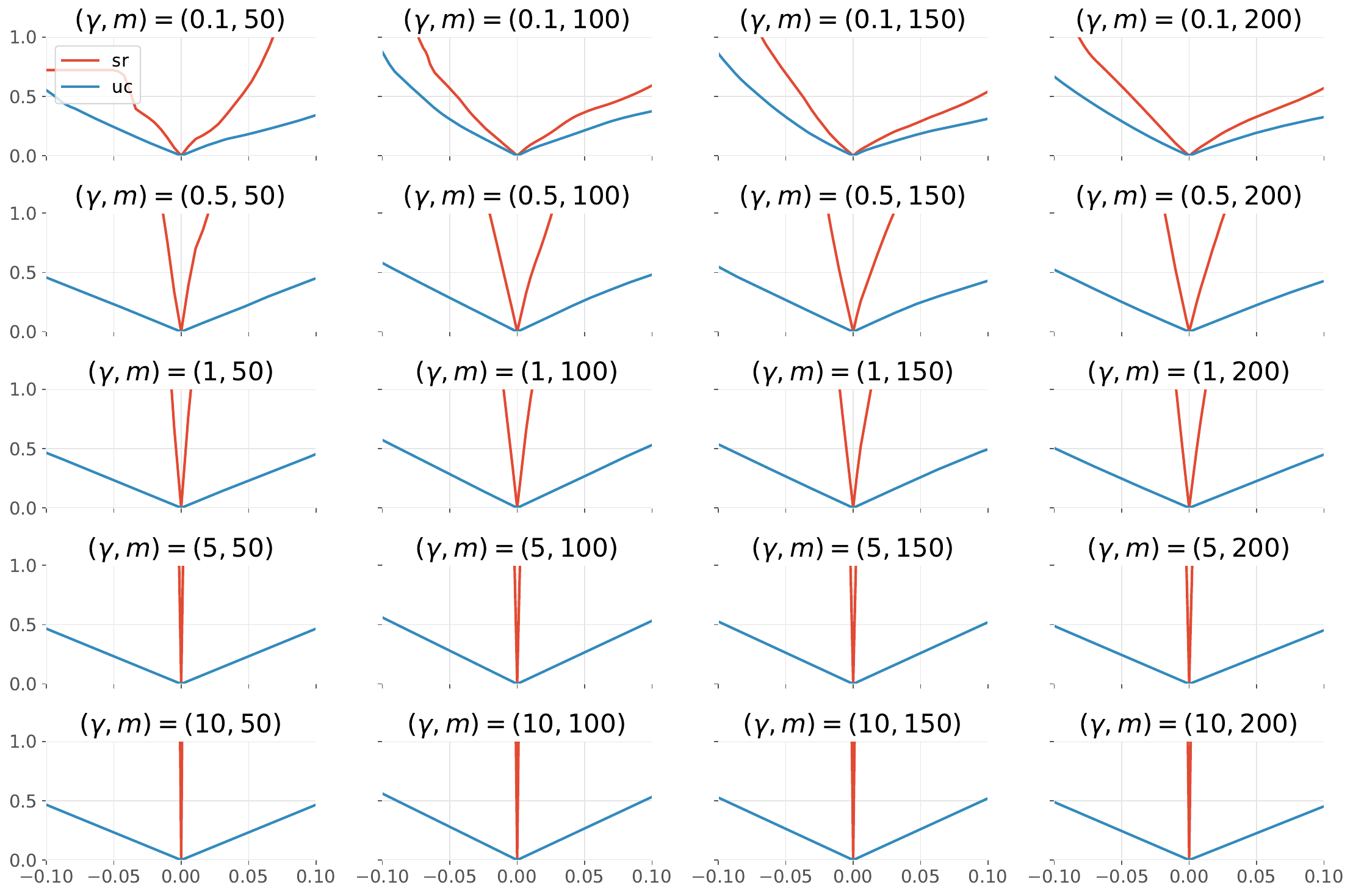}
\caption{Lipschitzness: $\lambda - \bar\lambda_{\text{best}}$ \emph{vs}.\
  $\| \bar{x}(\lambda) - \bar{x}(\bar\lambda_{\text{best}})\|$.\label{fig:noisescale-effect-on-lipschitzness}}
  \end{subfigure}
  
  \begin{subfigure}{0.95\linewidth}
    \centering
    \includegraphics[width=\textwidth]{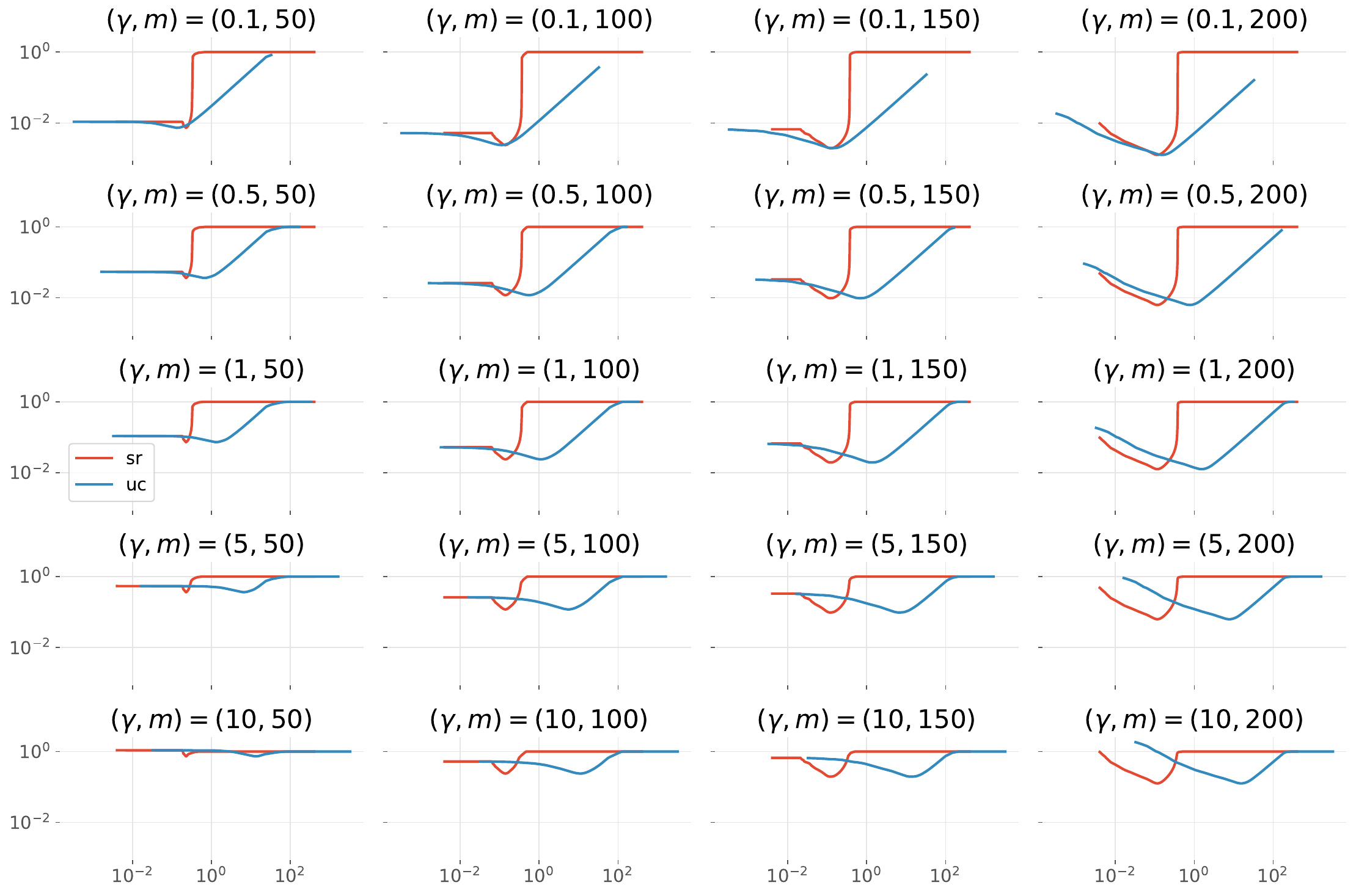}
    \caption{Relative error as a function of $\lambda$
      with log-log axis scaling.\label{fig:parameter-sensitivity-by-noisescale}}
  \end{subfigure}
  \caption{Effect of noise scale on error sensitivity for~\cref{eq:SR-LASSO}
    (sr) and \cref{eq:unconstrained-LASSO} (uc) faceted by
    $(\gamma, m) \in \{0.1, 0.5, 1, 5, 10\} \times \{50, 100, 150, 200\}$ with
    $(s, n) = (7, 200)$.}
  \label{fig:noisescale-experiment}
\end{figure}
In particular, we first investigate empirical Lipschitz behavior of the solution
function for~\cref{eq:SR-LASSO}
(see~\cref{fig:noisescale-effect-on-lipschitzness}). Again, as a reference we
compare with that for~\cref{eq:unconstrained-LASSO}. In addition, we numerically
examine the parameter sensitivity of the relative error for
both~\cref{eq:SR-LASSO} and~\cref{eq:unconstrained-LASSO}
(see~\cref{fig:parameter-sensitivity-by-noisescale}). In both cases, this is
done about the empirically optimal parameter values
$\bar{\lambda} = \bar\lambda_{\text{best}}$. Below, we set $\Lambda := 501$. The
logarithmically spaced grid $(\lambda_{i} / \lambda_{*})$ ranges from $10^{-3}$
to $100$ (and includes the point $1$). In this experiment we fix
$(s, n) = (7, 200)$ and use $\gamma \in \{0.1, 0.5, 1, 5, 10\}$,
$m \in \{50, 100, 150, 200\}$. Plotted results are depicted in $5 \times 4$
grids with each grid cell corresponding to a $(\gamma, m)$ pair.

We readily observe from~\cref{fig:noisescale-effect-on-lipschitzness} that for
any selected pairing of $(\gamma, m)$, the empirical Lipschitzness of
\cref{eq:SR-LASSO} is worse than that of
\cref{eq:unconstrained-LASSO}. Interestingly, we observe that increasing noise
scale tends to worsen the empirical Lipschitz behavior of \cref{eq:SR-LASSO},
while it remains (locally) similar for \cref{eq:unconstrained-LASSO} about the
selected reference value. 

In~\cref{fig:parameter-sensitivity-by-noisescale} we compare the relative errors
of each program as a function of $\lambda$. We observe
in~\cref{fig:parameter-sensitivity-by-noisescale} that, for $m$ fixed,
$\bar{\lambda}_{\text{best}}^{\text{SR}}$ is generally less sensitive to
variation in $\gamma$ than is $\bar{\lambda}_{\text{best}}^{\text{UC}}$. This
observation is consistent with the ``tuning robustness'' property characteristic
of~\cref{eq:SR-LASSO}
(\emph{cf.}~\cref{fig:motivation-known}). From~\cref{fig:noisescale-experiment},
we observe for all choices of $(\gamma, m)$ that~\cref{eq:SR-LASSO} is more
sensitive to its parameter choice than~\cref{eq:unconstrained-LASSO}, again
consistent with a comparison of the Lipschitz upper bounds
(\emph{cf.}~\cref{sec:theoretical-comparison}).

\section{Numerical investigation of our \srlasso{} theory}
\label{sec:numerics}

We present numerical simulations supporting the theoretical results of the
previous sections pertaining to solution uniqueness and local Lipschitz
moduli. Specifically, we examine the satisfiability of~\cref{ass:weak-1}
in~\cref{sec:numerics-uniqueness-sufficiency} with a graphical demonstration
of~\cref{th:sufficient-condition-for-uniqueness}, visualizing for a given set of
parameters when \cref{th:sufficient-condition-for-uniqueness}(ii) holds as a
function of $\lambda$. We investigate the tightness of the Lipschitz
bound~\cref{eq:lipschitz_ub_lamda} under~\cref{ass:Strong-1}
in~\cref{sec:empir-invest-lipsch-ub}. Refer to~\cref{sec:implementation-details}
for an overview of implementation details and relevant notation. For greater
detail beyond this, refer to our
\href{https://github.com/asberk/srlasso_revolutions}{code
  repository}~\cite{coderepo}.

\subsection{Empirical investigation of uniqueness sufficiency}
\label{sec:numerics-uniqueness-sufficiency}

We begin with an empirical investigation of when the sufficient conditions for
uniqueness hold, serving to establish an intuitive understanding of the
behavior underlying~\cref{th:sufficient-condition-for-uniqueness}. To this end, 
we fix a dual pair $(\bar x, \bar y)$ for~\cref{eq:SR-LASSO} (\ie{} $\bar{x}$
solves~\cref{eq:SR-LASSO} and $\bar{y}$ solves~\cref{eq:srlasso-dual};
see~\cref{prop:FR}) and numerically solve the convex program
\begin{align}
  \label{eq:auxiliary-problem}
  \min_{z \in \reals^{m}} \|A_{I^{C}}^{\top} (\bar y + z)\|_{\infty}  %
  \text{\quad s.t.\quad} %
     [A_I \; \bar y]^\top z = 0.
\end{align}
We denote the optimal value of the program by $Z^{*}$ and any solution to the
program by $\bar z$. Then, \cref{ass:weak}(ii) is satisfied if
$Z^{*} < \lambda$. We visualize $Z^{*}$ as a function of $\lambda$
in~\cref{fig:uniqueness-sufficiency-simple}
\begin{figure}[t]
  \centering
  \includegraphics[width=.8\textwidth]{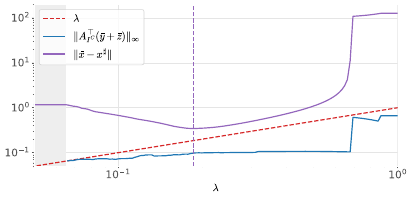}
  \caption{Visualizing uniqueness sufficiency for
    $(m, n, s, \gamma) = (100, 200, 2, 0.1)$. Upper solid line: error
    $\|\bar x(\lambda) - x^\sharp\|$ where $\bar x(\lambda)$
    solves~\cref{eq:SR-LASSO}. Lower solid line: empirical version
    of~\cref{ass:weak}(ii) that partially suffices for uniqueness,
    $\|A_{I^{C}}^{\top}\bar y\|_{\infty}$, where $\bar y$
    solves~\cref{eq:srlasso-dual} and $\bar z$
    solves~\cref{eq:auxiliary-problem}. Grey shaded vertical rectangles
    correspond with $\lambda$ for which $Z^{*} = \infty$. Diagonal dashed line
    $y = \lambda$ serves as reference for lower solid line. Horizontal position
    of vertical dashed line denotes
    $\bar{\lambda}^{\text{SR}}_{\text{best}}$.\label{fig:uniqueness-sufficiency-simple}}
\end{figure}
by plotting $Z^{*} = Z^{*}(\lambda)$ and the diagonal line ``$y =
\lambda$''. The former line is given by the lower solid line, the latter by the
diagonal dashed line. The upper solid line corresponds to the error
$\|\bar x(\lambda) - x^\sharp\|$. The horizontal position of the vertical
dashed line indicates $\bar{\lambda}_{\text{best}}^{\text{SR}}$ The plot is
shown on a log-log scale.  Above, the numerical solution for $Z^{*}$ was
computed using \cvxpy{} v1.2~\cite{agrawal2018rewriting,diamond2016cvxpy} with
the \mosek{} solver~\cite{mosek}. Values of $\lambda$ for which $Z^{*} = \infty$
are shown as grey shaded vertical rectangles. The relative error between the
primal~\cref{eq:SR-LASSO} and dual~\cref{eq:srlasso-dual} optimal values was
$8.88\times 10^{-7}$, meaning that the two are comfortably within numerical
tolerance, given the optimization parameter settings.

In addition, we present a heatmap in~\cref{fig:uniqueness-sufficiency-heatmap}
\begin{figure}[t]
  \centering
  \includegraphics[width=\textwidth]{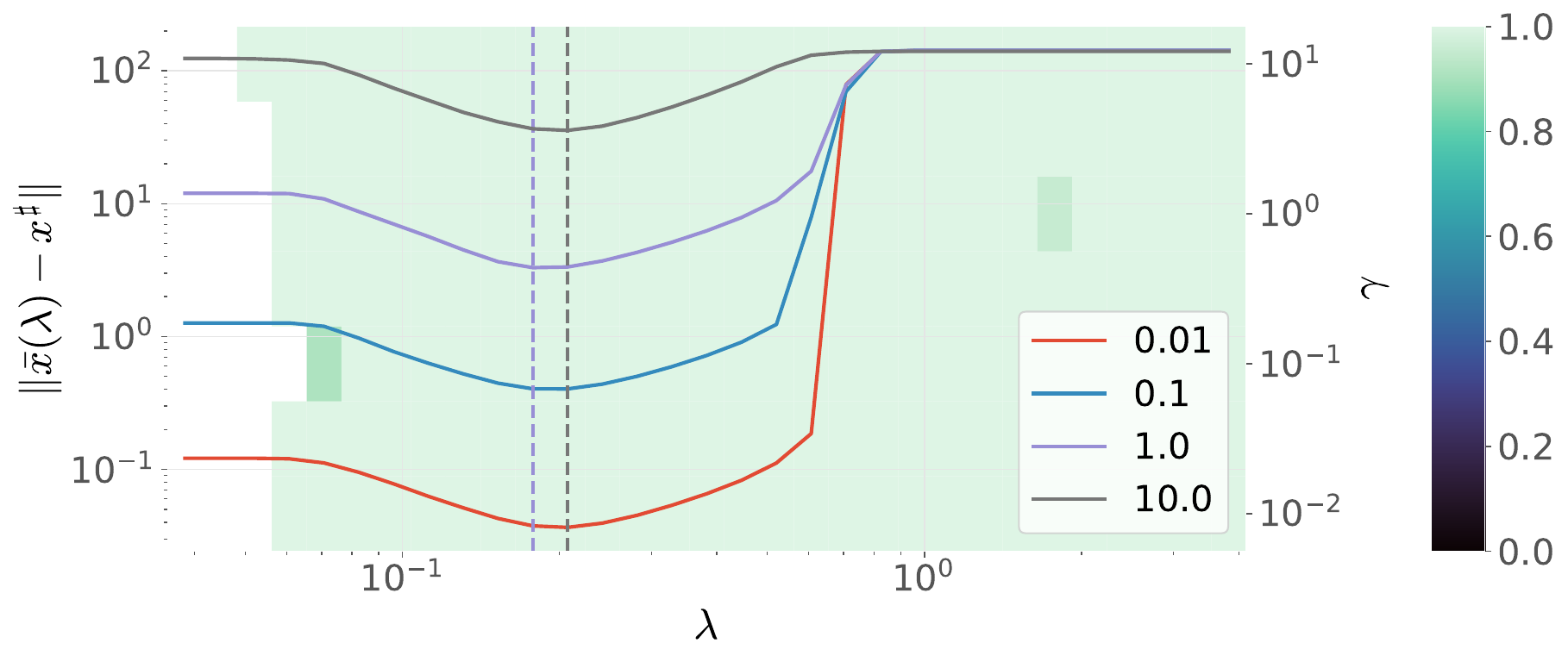}
  \caption{Background: uniqueness sufficiency heatmap displaying the proportion
    of $20$ independent trials for which
    $\|A_{I^{C}}^{\top} (\bar y + \bar z)\|_{\infty} < \lambda$ is satisfied,
    where $\bar y$ solves~\cref{eq:srlasso-dual} and $\bar z$
    solves~\cref{eq:auxiliary-problem}. Voxels correspond to $\lambda$
    (horizontal axis) and noise scale $\gamma$ (right axis). White regions: no
    data. Foreground: error $\|\bar x(\lambda) - x^\sharp\|$ (left axis) as
    a function of $\lambda$ for four choices of the noise scale $\gamma$ (see
    legend). Vertical dashed lines are drawn at $\bar\lambda_{\text{best}}$ for each. }
  \label{fig:uniqueness-sufficiency-heatmap}
\end{figure}
demonstrating the relative frequency that the sufficient condition for
uniqueness is satisfied for a range of $31 \times 7$ logarithmically spaced
parameter values $(\lambda, \gamma) \in [0.1, 10] \times [0.01, 10]$ (horizontal
and right axis, respectively). Each pixel displays a mean of $20$ independent
repetitions with $1$ corresponding to the sufficient condition being satisfied
for all trials; $0$ corresponding to the condition being satisfied for none of
the trials. Apart from the changing noise scale $\gamma$, the signal/measurement
model is the same as described above. We also compute
$Z^{*} = Z^{*}(i, \gamma, \lambda)$ as described above, where $i \in [20]$ is
the trial number. White regions in the heatmap correspond none of the $20$
trials yielding an inexact solution, violating a tenet of our theory that
$A \bar{x} \neq b$. Superposed on the heatmap is a plot of the recovery error (left vertical axis) as
a function of $\lambda$ for $4$ of the values of $\gamma$ (see
legend). \cref{fig:uniqueness-sufficiency-heatmap} reveals a
sizable region where the sufficient condition for uniqueness is
empirically assured. Moreover, this region encompasses all
$\bar{\lambda}_{\text{best}}$ values with a comfortable margin, and it is relatively insensitive to $\gamma$.

\subsection{Empirical investigation of Lipschitz upper bound}
\label{sec:empir-invest-lipsch-ub}


Finally, we compare the Lipschitz upper bound~\cref{eq:lipschitz_ub_lamda} to
the empirical Lipschitz quantity
$\| \bar{x}(\lambda) - \bar{x}(\bar{\lambda}) \|$ where
$\bar{\lambda} := \bar{\lambda}_{\text{best}}$. To this end, we investigate two
settings where the dimensional parameters are varied, with results displayed
in~\cref{fig:empirical-lipschitz-bound}.
\begin{figure}
  \centering
  \begin{subfigure}[c]{.95\linewidth}
    \centering
    \includegraphics[width=\linewidth]{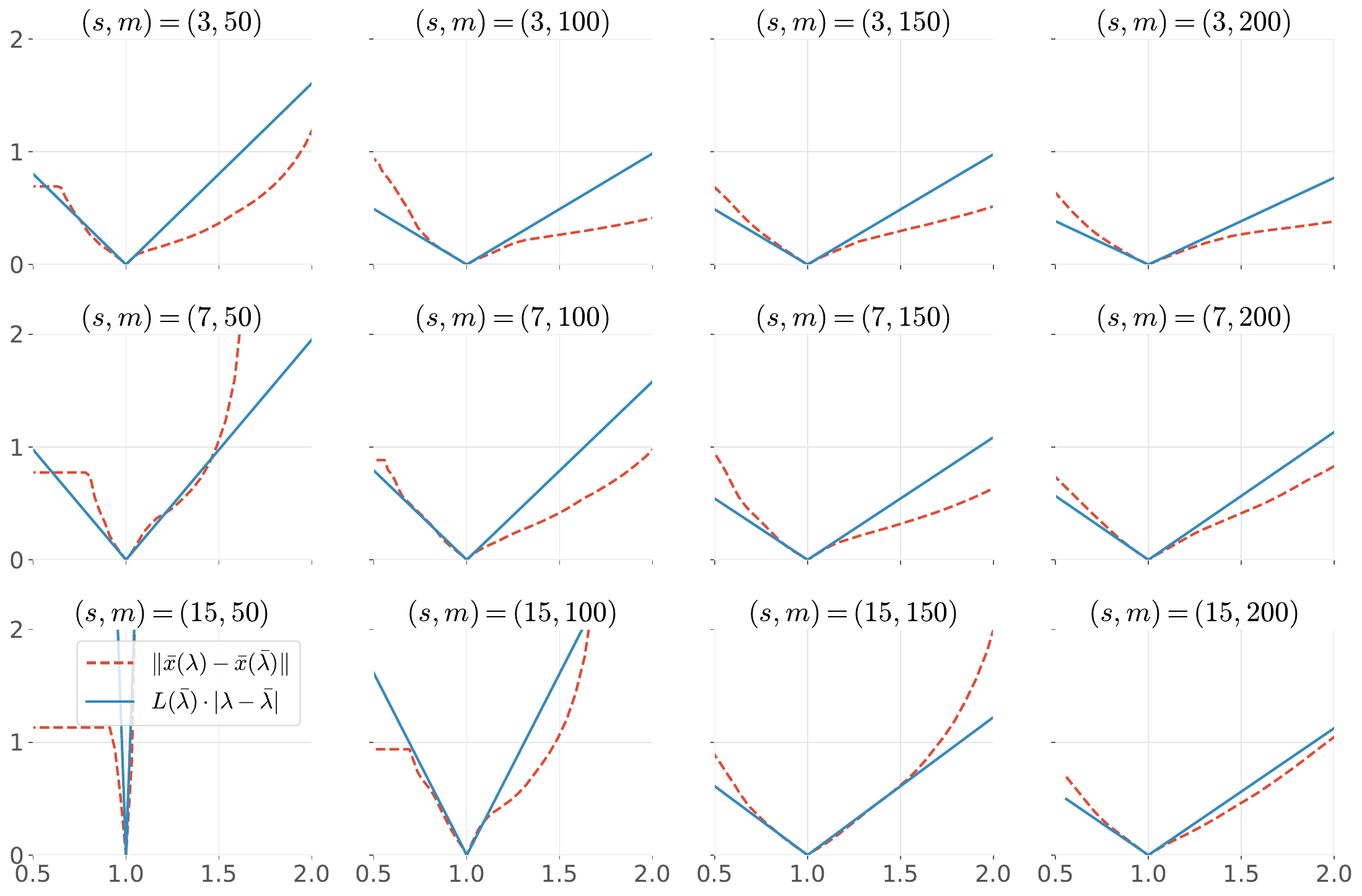}
    \caption{Variation in $(m, s)$.\label{fig:empirical-lipschitz-bound-ms}}
  \end{subfigure}

  \begin{subfigure}[c]{.95\linewidth}
    \centering
    \includegraphics[width=\linewidth]{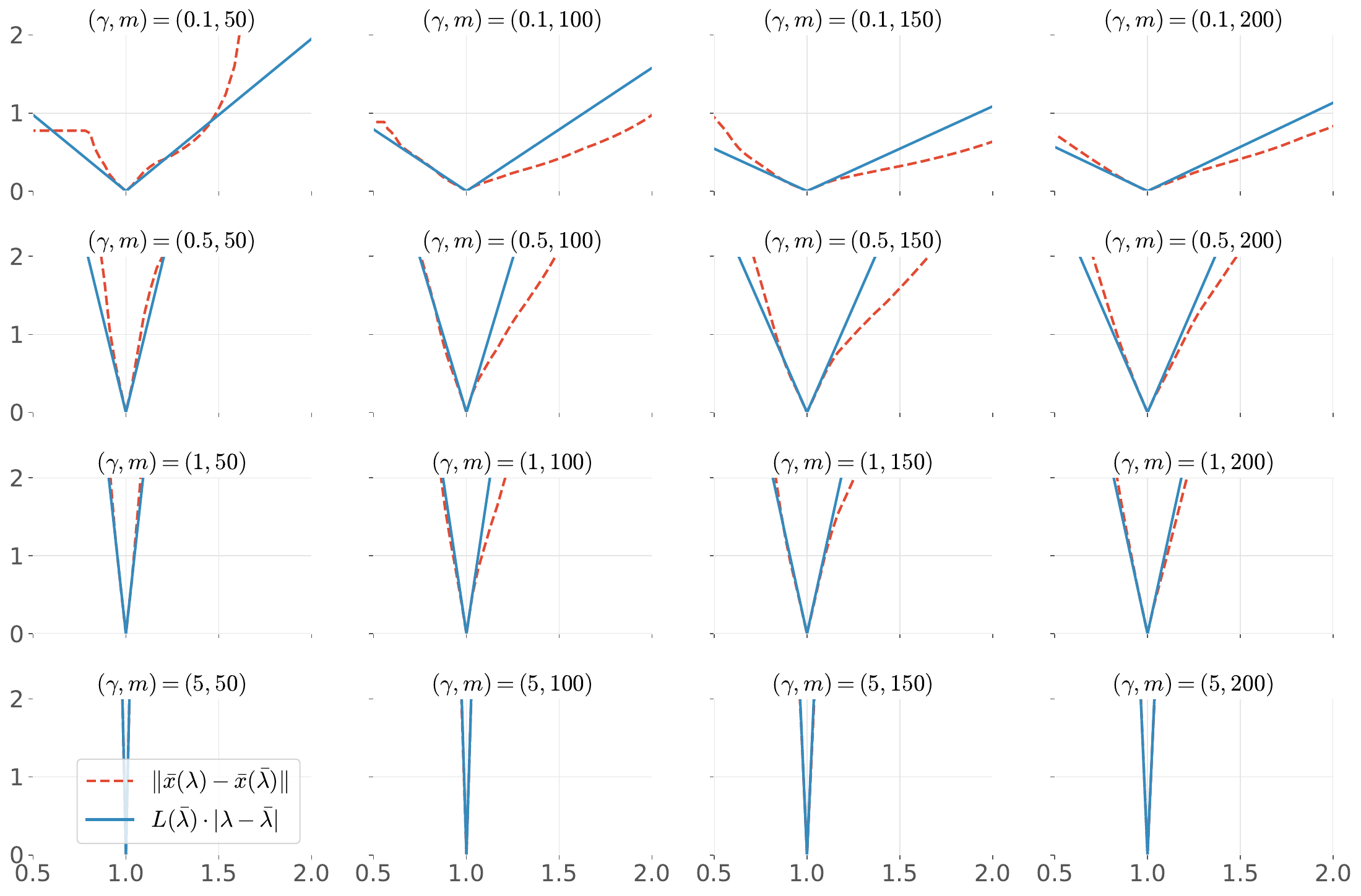}
    \caption{Variation in $(m, \gamma)$.\label{fig:empirical-lipschitz-bound-mg}}
  \end{subfigure}
  \caption{Effect of varying dimensional parameters on the Lipschitz upper
    bound: $L(\bar{\lambda}) |\lambda - \bar{\lambda}|$ \emph{vs}.\ the
    empirical Lipschitz quantity
    $\| \bar{x}(\lambda) - \bar{x}(\bar{\lambda}) \|$ as a function of
    $\lambda_{\text{nmz}}$. $L(\bar{\lambda})$ is computed as
    in~\cref{eq:lipschitz_ub_lamda}. \textbf{Top}:
    $(m,s) \in \{50, 100, 150, 200\} \times \{3, 7, 15\}$; \textbf{Bottom}:
    $(m, \gamma) \in \{50, 100, 150, 200\} \times \{0.1, 0.5, 1,
    5\}$. \label{fig:empirical-lipschitz-bound}}
\end{figure}
Throughout, we choose $n = 200$. In the first experiment, shown
in~\cref{fig:empirical-lipschitz-bound-ms}, we set $\gamma = 0.1$ and choose
$(m,s) \in \{50, 100, 150, 200\} \times \{3, 7, 15\}$; in the second, shown
in~\cref{fig:empirical-lipschitz-bound-mg}, we set $s = 7$ and choose
$(m,\gamma) \in \{50, 100, 150, 200\} \times \{0.1, 0.5, 1, 5\}$. Generally, we
observe that the Lipschitz upper bound $L(\bar{\lambda})$ given
by~\cref{eq:lipschitz_ub_lamda} is a tight local approximation to the true
Lipschitz behavior of the solution $\bar{x}(\lambda)$ about $\bar{\lambda}$.

\section{Conclusion}
\label{sec:conclusion} 

In this paper, we studied the Square Root LASSO (\srlasso{}) \cref{eq:SR-LASSO}.
We established sufficient conditions for its well-posedness, namely
\cref{ass:weak}, and linked it to two stronger regularity conditions, the
intermediate condition \cref{ass:intermediate} and the strong condition
\cref{ass:Strong}, respectively. The intermediate condition is shown to imply
local Lipschitzness and directional differentiability of the solution map as a
function of the right-hand side and the tuning parameter (around a reference
point); the strong condition, in turn, guarantees continuous differentiability
of said solution map. We then leveraged these results to compare the \srlasso{}
to its close relative, the (unconstrained) \lasso{} from a theoretical
perspective. This comparison suggests that the celebrated robustness of optimal
parameter tuning to noise of the SR-LASSO comes at the price of elevated
sensitivity of the solution map to the tuning parameter itself. Our numerical
experiments confirmed the presence of this robustness-sensitivity trade off for
parameter tuning, and illustrated the sharpness of our Lipschitz bounds and the
validity of the main assumptions upon which the theory relies on.

We conclude by discussing possible extensions of this work and open problems:
although we focused on the dependence of the solution map in $b$ and $\lambda$,
we point out that it is straightforward (at the cost of more computational
overhead) to extend all stability results to the case where also the design
matrix $A$ is a parameter.  An obvious question is whether the Lipschitz stability results presented in \cite{berk2023lasso} and here can be extended beyond the LASSO and SR-LASSO framework, respectively. If one is merely interested in a qualitative result, the answer is clearly positive. Several papers in the literature address this for certain regularizers and $\ell_2$-fidelity term, see, e.g., \cite{bolte2024differentiating} and \cite{nghia2024geometric}. On the other hand, getting quantitative  results, i.e., explicit local Lipschitz constants, is a topic of ongoing research.

Moreover, similarly to~\cite{berk2023lasso}, our
results could be explicitly applied to compressed sensing theory by combining
our Lipschitz bounds with explicit estimates of the sparsity of SR-LASSO
solutions~\cite{foucart2023sparsity}. 
Our numerical experiments in \cref{sec:numerics-uniqueness-sufficiency} show that \cref{ass:weak} is satisfied with high probability in the case of random Gaussian sampling for a wide range of $\lambda$'s (including the optimal one) and for various levels of noise corrupting the measurements. However, theoretically showing the validity of \cref{ass:weak} for suitable classes of sampling matrices remains an open problem. In addition, whether the standing assumption that $b\notin \rge A_I$ is necessary throughout is, 
at this point, unclear, but the variational-analytic approach builds fundamentally on a sum rule for the graphical derivative and coderivative that would not be applicable in this form without the made assumption. Finally, in the LASSO case it is
known~\cite{zhang2015necessary} that the analogous sufficient condition to
\cref{ass:weak} is also necessary for uniqueness. At this point, whether this is
also true for the SR-LASSO, is an open question and we challenge the reader to
clarify it.

\appendix 

\section{Proof of shrinking property}
\label{app:Shrinking}

\noindent
Here, we furnish proof of~\cref{lem:shrinking-property}.  Note that, for $f, g : \reals^{n} \to \rbar$, their infimal convolution~\cite{rockafellar1970convex} is denoted by $(f \mathbin{\hash} g)(x) := \inf_{u} f(x - u) + g(u)$.

\begin{proof}[{Proof of~\cref{lem:shrinking-property}}]
  The (primal) problem defining $p^{*}$ reads
  \begin{align}
    \label{eq:fix--primal}
    \min_{z \in \reals^{m}} \| B^{\top}(\bar{y} + z) \|_{\infty} + \delta_{\varepsilon \mathbb{B} \cap \mathcal{T}} (z). 
  \end{align}
  The minimum is attained due to compactness and lower semicontinuity. Now, set
  $ f := \delta_{\varepsilon \mathbb{B}\cap \mathcal{T}}$ and
  $ g := \| B^{\top}\bar{y} + (\cdot) \|_{\infty}.$ With
  $\rmd_{\mathcal T^\perp}$, the Euclidean distance to $\cT^\perp$, we find
  \begin{align*}
    f^{*}(u) %
    &= \delta_{\varepsilon \mathbb{B} \cap \mathcal{T}}^{*}(u) %
    \\
    & = (\delta_{\varepsilon \mathbb{B}} + \delta_{\mathcal{T}})^{*}(u) %
    \\
    & = (\sigma_{\varepsilon \mathbb{B}} \mathbin{\hash} \sigma_{\mathcal{T}})(u) %
    \\
    & = \inf_{y} \varepsilon \| y \| + \delta_{\mathcal{T}^{\perp}}(u - y) %
    \\
    & = \varepsilon \rmd_{\mathcal{T}^{\perp}}(u),
  \end{align*}
  where in the third line, we have used~\cite[Theorem 16.4]{rockafellar1970convex} combined with the fac that 
  $0 \in (\inter \varepsilon \mathbb{B})\cap\mathcal{T}$. We also have
  $g^{*}(u) = \delta_{\mathbb{B}_{1}}(u) - \ip{B^{\top}\bar{y}}{u}$. Hence, with
  $\phi(u) := - \ip{B^{\top}\bar{y}}{u} - \varepsilon
  \rmd_{\mathcal{T}^{\perp}}(Bu)$, the (Fenchel-Rockafellar) dual problem of~\cref{eq:fix--primal} is
  \begin{align}
    \label{eq:fix--dual}
    \max - f^{*}(B u) - g^{*}(-u) %
    \qquad \iff \qquad %
    \max_{u \in \mathbb{B}_{1}} \phi(u).
  \end{align}
  Now, observe that (see \eg{} \cite{horn2012matrix})
  $
    -\ip{B^{\top} \bar{y}}{u} \leq \| B^{\top} \bar{y} \|_{\infty} \| u \|_{1}.
  $
  Hence, for every feasible point of~\cref{eq:fix--dual}, we have
  \begin{align}
    \label{eq:fix--observation-1}
    \phi(u) \leq \| B^{\top}\bar{y} \|_{\infty} - \varepsilon \rmd_{\mathcal{T}^{\perp}}(Bu).
  \end{align}
  We claim that $\phi(u) < \| B^{\top}\bar{y} \|_{\infty}$ for all
  $u \in \mathbb{B}_{1}$. Indeed, assume to the contrary the existence of a
  $\hat u$ feasible for \cref{eq:fix--dual} such that
  $\phi(\hat u) = \| B^{\top} \bar{y} \|_{\infty}$. Then
  \cref{eq:fix--observation-1} implies
  \begin{align*}
    \rmd_{\mathcal{T}^{\perp}}(B\hat u) = 0 %
    \qquad &\iff \qquad %
    B\hat u \in \mathcal{T}^{\perp} = \reals \cdot \{\bar{y}\} + \rge C %
    \\
           & \implies \qquad %
    \bar{y} \in \rge C + \rge B = \rge [B\; C],
  \end{align*}
  contradicting an assumption of the lemma. Consequently, since the dual problem
  admits a solution, $\bar{u}$ say, we find by strong duality that
$p^{*} = \phi(\bar{u}) < \| B^{\top}\bar{y} \|_{\infty}.$ 
\end{proof}

\section{Proof of analytic solution formula under~\cref{ass:intermediate}}
\label{app:Explicit}

\noindent
Here, we provide the proof for the analytic expression for the (unique) solution
under the intermediate condition from~\cref{ass:intermediate}.

\begin{proof}[Proof of~\cref{prop:Explicit}]
  By assumption, the dual problem~\cref{eq:srlasso-dual} has a unique solution
  $ \bar{y} = \frac{b - A\bar{x}}{\| A\bar{x} - b \|}.$ Therefore, using the
  optimality conditions for $\bar{x}$, there exists a unique subgradient
  $v \in \partial \| \cdot \|_{1}(\bar{x})$ such that
  $A^{\top} \bar{y} = \lambda v$. By definition of $v$ and the fact that
  $I \subseteq J$, we have
  $\| \bar{x} \|_{1} = \ip{v}{\bar{x}} = \ip{v_{J}}{\bar{x}_{J}}$.  We thus
  rewrite the strong duality expression~\cref{prop:FR}(b)(ii) as
  \begin{align}
    \label{eq:intermed-1}
    \| A\bar{x} - b \| = \ip{b}{\bar{y}} - \lambda \ip{v_{J}}{\bar{x}_{J}}. 
  \end{align}
  Using~\cref{cor:Invar}(b), we can rewrite the optimality conditions as
  $A^{\top} (b - A\bar{x}) = \lambda \| A\bar{x} - b \| v$.  Restricting to $J$ and using~\cref{eq:intermed-1} gives
 \[
    A_{J}^{\top} (b - A_{J}\bar{x}_{J}) %
    = \lambda \left( \ip{b}{\bar{y}} - \lambda \ip{v_{J}}{\bar{x}_{J}} \right)v_{J} %
       = \lambda \ip{b}{\bar{y}} v_{J} - \lambda^{2} v_{J}v_{J}^{\top} \bar{x}_{J}. 
\]
  After rearranging, we obtain
  $\left( A_{J}^{\top} A_{J} - \lambda^{2} v_{J}v_{J}^{\top} \right) \bar{x}_{J} %
    = A_{J}^{\top} b - \lambda \ip{b}{\bar{y}} v_{J}$.
  It remains to verify that the matrix
  $A_{J}^{\top} A_{J} - \lambda^{2} v_{J}v_{J}^{\top}$ satisfies
  the desired identity and is invertible. To this end, the $J$-restricted
  optimality conditions $A_{J}^{\top} \bar{y} = \lambda v_{J}$ imply
  \begin{align*}
    A_{J}^{\top} A_{J} - \lambda^{2} v_{J}v_{J}^{\top} %
    = A_{J}^{\top} \left( \identity - \bar{y}\bar{y}^{\top} \right) A_{J}. 
  \end{align*}
  To obtain invertibility of this matrix, we apply \cref{lem:SM}, using that
  $A_{J}$ has full column rank and $\bar{y} \notin \rge A_{J}$ (because
  $b \notin \rge A_{J}$). In particular,
  \begin{align*}
    \bar{x}_{J} %
    = B \left( A_{J}^{\top} b - \lambda \ip{b}{\bar{y}} v_{J} \right), %
    \qquad%
    B := \left[ A_{J}^{\top}(\identity - \bar{y}\bar{y}^{\top}) A_{J} \right]^{-1}.
  \end{align*}
  An explicit expression for $B$ is provided by~\cref{lem:SM}. Notice that
  uniqueness of $\bar{y}$ implies that of $v$ and $J$, and hence too of
  $x_{J}$. Finally, $x_{J^{C}} \equiv 0$ because $I \subseteq J$. Hence, \[\bar{x} %
   = L_{J} \left( B \left( A_{J}^{\top} b - \ip{b}{\bar{y}} A_{J}^{\top} \bar{y} \right) \right)
       = L_{J} \left( B A_{J}^{\top} (\identity - \bar{y} \bar{y}^{\top}) b \right).
\]
\end{proof}

\section{Auxiliary results}
\label{sec:auxiliary-results}

\begin{lemma}
  \label{lem:phi}
  Let $\bar A\in \R^{m\times n}$, $\bar b\in \R^m$, $\bar x\in \R^n$ such that
  $\bar A\bar x\neq \bar b$ and let $\bar \lambda > 0$. Then, there exists a
  neighborhood $U$ of $(\bar A, \bar b,\bar \lambda, \bar x)$ such that the
  function $\phi:U\to \R^n$,
  \[
    \phi(A,b,\lambda,x) := \frac{A^{\top}(Ax-b)}{\lambda\|Ax-b\|},
  \]
  is well defined and continuously differentiable on $U$ with partial
  derivatives:
  \begin{itemize}
  \item[(a)] $D_A \phi(A,b,\lambda, x)(W)=\frac{1}{\lambda\|Ax-b\|}\left[(A^{\top}W+W^{\top}A)x-W^{\top}b-\frac{A^{\top}(Ax-b)(Ax-b)^{\top}Wx}{\|Ax-b\|^2} \right]$;
  \item[(b)] $D_b \phi(A,b,\lambda,x)=-\frac{1}{\lambda\|Ax-b\|}\left[ \frac{A^{\top}-A^{\top}(Ax-b)(Ax-b)^{\top}}{\|Ax-b\|^2}\right]$;
  \item[(c)] $D_\lambda \phi(A,b,\lambda,x)=-\frac{1}{\lambda^2 \|Ax-b\|}A^{\top}(Ax-b)$;
  \item[(d)] $D_x  \phi(A,b,\lambda,x)=\frac{1}{\lambda\|Ax-b\|}\left[A^{\top}A-A^{\top}\frac{Ax-b}{\|Ax-b\|}\frac{(Ax-b)^{\top}}{\|Ax-b\|}A\right]$.
  \end{itemize}
\end{lemma}

\bibliographystyle{siamplain}
\bibliography{bibliography}

\end{document}

%% file: shared.tex
\usepackage{lipsum}
\usepackage{amsfonts, amsmath}
\usepackage{graphicx}
\usepackage{epstopdf}
\usepackage{algorithmic}
\usepackage{amssymb}
\usepackage{cleveref}
\ifpdf
  \DeclareGraphicsExtensions{.eps,.pdf,.png,.jpg}
\else
  \DeclareGraphicsExtensions{.eps}
\fi

\newsiamremark{example}{Example}

\newcommand{\newassumption}[2]{
  \theoremstyle{plain}
  \theoremheaderfont{\normalfont\sc}
  \theorembodyfont{\normalfont\itshape}
  \theoremseparator{.}
  \theoremsymbol{}
  \newtheorem{#1}{#2}
}

\newassumption{assumption}{Assumption}

\newsiamremark{remark}{Remark}
\newsiamremark{hypothesis}{Hypothesis}
\crefname{hypothesis}{Hypothesis}{Hypotheses}
\newsiamthm{claim}{Claim}

\newcommand{\aff}{\mathrm{aff}\,}
\def\argmin{ \mathop{{\rm argmin}}}

\newcommand{\dom}{\mathrm{dom}\,}
\newcommand{\epi}{\mathrm{epi}\,}

\newcommand{\gph}{\mathrm{gph}\,}

\DeclareMathOperator{\inter}{int}

\newcommand{\ri}{\mathrm{ri}\,}

\newcommand{\para}{\mathrm{par}\,}
\newcommand{\rge}{\mathrm{rge}\,}
\newcommand{\sgn}{\mathrm{sgn}}

\newcommand{\lin}{\mathrm{span}\,}

\def\Limsup{\mathop{{\rm Lim}\,{\rm sup}}}

\newcommand{\p}{\partial}

\newcommand{\R}{\mathbb{R}}

\newcommand{\supp}{\mathrm{supp}}

\newcommand{\rank}{\mathrm{rank}\,}
\newcommand{\rbar}{\overline{\mathbb R}}

\newcommand{\rp}{\mathbb R\cup\{+\infty\}}

\newcommand{\IFF}{\quad\Longleftrightarrow\quad}

\newcommand{\bB}{\mathbb{B}}
\newcommand{\bN}{\mathbb{N}}
\newcommand{\bE}{\mathbb{E }}

\newcommand{\ip}[2]{\left\langle #1,\, #2\right\rangle}

\newcommand{\set}[2]{\left\{#1\,\left\vert\; #2\right.\right\}}

\newcommand{\sig}{\sigma}

\newcommand{\lam}{\lambda}

\newcommand{\cS}{\mathcal{S}}
\newcommand{\cT}{\mathcal{T}}

\newcommand{\cV}{\mathcal{V}}

\newcommand{\AND}{\ \mbox{ and }\ }
\newcommand{\st}{\ \mbox{s.t.}\ }

\newcommand{\reals}{\ensuremath{\mathbb{R}}}
\newcommand{\nats}{\ensuremath{\mathbb{N}}}

\DeclareMathOperator{\iid}{\overset{\text{iid}}{\sim}}
\DeclareMathOperator{\proj}{P}

\DeclareMathOperator{\identity}{\mathbb{I}}
\DeclareMathOperator{\rmd}{d}

\definecolor{codegray}{gray}{0.95}
\newcommand{\code}[1]{\fcolorbox{codegray}{codegray}{\small\texttt{#1}}}
\newcommand{\cvxpy}{\code{CVXPY}}

\newcommand{\mosek}{\code{MOSEK}}

\newcommand{\eg}{\emph{e.g.,}}

\newcommand{\etal}{\emph{et al.}}
\newcommand{\ie}{\emph{i.e.,}}
\newcommand{\vs}{\emph{vs.\@}}

\newcommand{\lasso}{LASSO}
\newcommand{\srlasso}{SR-LASSO}

\headers{Square Root LASSO}{A. Berk, S. Brugiapaglia, and T. Hoheisel}

\title{Square Root {LASSO}: well-posedness,
  Lipschitz stability and the tuning trade off\thanks{\today\funding{The first author was partially supported by a postdoc stipend from the Centre de Recherche Mathématiques (CRM) as well as the Institut de valorisation des données (IVADO) and NSERC. The second author acknowledges the support of NSERC through grant RGPIN-2020-06766, the Fonds de Recherche du Québec Nature et Technologies
(FRQNT) through grant 313276, the Faculty of Arts and Science of Concordia University and the CRM. The third author was partially supported by the NSERC discovery grant RGPIN-2017-04035.}}}

\author{Aaron Berk\thanks{Department of Mathematics and Statistics, McGill University\newline\indent{}\indent{}(\email{aaron.berk@mcgill.ca}, \email{tim.hoheisel@mcgill.ca}).}
\and Simone Brugiapaglia\thanks{Department of Mathematics and Statistics, Concordia University\newline\indent{}\indent{}(\email{simone.brugiapaglia@concordia.ca}).}
\and Tim Hoheisel\footnotemark[2]}

\usepackage{amsopn}
\usepackage{mathabx}
\usepackage{subcaption}
\usepackage{tikz-cd}

\usepackage{listings}
\usepackage{inconsolata}
\usepackage{newfloat}
\DeclareFloatingEnvironment[
    fileext=loalg,
    listname={List of Algorithms},
    name=Algorithm,
    placement=tbhp,
    within=section,
    ]{lstalgorithm}

\crefname{lstalgorithm}{Algorithm}{Algorithms}

\makeatletter
\newcommand{\crefnames}[3]{%
  \@for\next:=#1\do{%
    \expandafter\crefname\expandafter{\next}{#2}{#3}%
  }%
}
\makeatother

\crefnames{section,subsection,subsubsection}{\S\!}{\S\!}
\usepackage{enumitem}
